\documentclass[11pt]{amsart}
\usepackage{amsthm, amsmath,amsfonts, amssymb, enumerate, textcmds, tensor, enumitem}
\usepackage{pst-node,bbm}
\usepackage{graphicx}
\usepackage{tikz-cd} 
\usepackage[margin=1.25in]{geometry}
\usepackage{graphicx,tikz,tikz-cd}
\usepackage{pinlabel}
\usepackage[colorinlistoftodos]{todonotes}

\newcommand{\R}{\mathbb R}

\newtheorem{thm}{Theorem}[section]
\newtheorem{lem}[thm]{Lemma}
\newtheorem{prop}[thm]{Proposition}

\newtheorem*{thma}{Theorem A}
\newtheorem*{thmb}{Theorem B}
\newtheorem*{thmc}{Theorem C}

\DeclareMathOperator{\Real}{Re}

\DeclareMathOperator{\E}{\mathcal{E}}

\DeclareMathOperator{\PSL}{\textrm{PSL}}

\DeclareMathOperator{\AdS}{\textrm{AdS}^3}

\theoremstyle{definition}
\newtheorem{defn}[thm]{Definition}
\newtheorem{remark}[thm]{Remark}
\begin{document}
\title[Almost strict domination and AdS 3-manifolds]{Almost strict domination and Anti-de Sitter 3-manifolds}
\author{Nathaniel Sagman}
\address{Nathaniel Sagman: University of Luxembourg, 2 Av. de l'Universite, 4365 Esch-sur-Alzette, Luxembourg.} \email{nathaniel.sagman@uni.lu}

\begin{abstract}
We define a condition called almost strict domination for pairs of representations $\rho_1:\pi_1(S_{g,n})\to \PSL(2,\mathbb{R})$, $\rho_2:\pi_1(S_{g,n})\to G$, where $G$ is the isometry group of a Hadamard manifold, and prove it holds if and only if one can find a $(\rho_1,\rho_2)$-equivariant spacelike maximal surface in a certain pseudo-Riemannian manifold, unique up to fixing some parameters. The proof amounts to setting up and solving an interesting variational problem that involves infinite energy harmonic maps. Adapting a construction of Tholozan, we construct all such representations and parametrize the deformation space. 

When $G=\PSL(2,\mathbb{R})$, an almost strictly dominating pair is equivalent to the data of an anti-de Sitter 3-manifold with specific properties. The results on maximal surfaces provide a parametrization of the deformation space of such $3$-manifolds as a union of components in a $\PSL(2,\mathbb{R})\times \PSL(2,\mathbb{R})$ relative representation variety.
\end{abstract}
\maketitle

\begin{section}{Introduction}
Recently there has been a growing interest in Lorentzian geometry in low dimensions, with a focus on Lorentzian space forms. Among the most important space forms is the anti-de Sitter space, which has constant curvature $-1$. This study was originally motivated by general relativity, where anti-de Sitter metrics are solutions to the Einstein equations, and now there are more modern applications in physics (see the paper of Witten \cite{Wi}, for instance). Another direction is the study of properly discontinuous group actions on Lorentzian space forms, or more generally on Clifford-Klein forms, which has its roots in some famous conjectures about affine geometry.  

Some aspects of the theory of properly discontinuous actions are well understood. However, in our previous paper \cite{S} we found anti-de Sitter structures on Seifert-fibered $3$-manifolds that fall outside the scope of the usual theory. This current paper has grown out of an attempt to elaborate on the situation from \cite{S}. In the end, we have found a general property for representations $\rho_1:\pi_1(S_{g,n})\to \PSL(2,\mathbb{R})$, $\rho_2:\pi_1(S_{g,n})\to G$, where $G$ is the isometry group of a Hadamard manifold $(X,\nu)$. 
\begin{subsection}{On AdS $3$-manifolds}
In dimension $3$, anti-de Sitter space (from now on, $\AdS$) identifies isometrically with the Lie group $\PSL(2,\mathbb{R})$, equipped with (a constant multiple of) its Lorentzian Killing metric. The space and time orientation preserving component of the isometry group is $\PSL(2,\mathbb{R})\times\PSL(2,\mathbb{R})$, where the action is by left and right multiplication: $$(g,h)\cdot x = gxh^{-1}.$$ An AdS $3$-manifold is a Lorentzian $3$-manifold of constant curvature $-1$. Equivalently, such a manifold is locally isometrically modelled on $\PSL(2,\mathbb{R})$.

If an AdS $3$-manifold is geodesically complete, meaning geodesics run for all time, then it comes from a proper quotient of $\widetilde{\PSL}(2,\mathbb{R})$ with respect to the lift of the action above. Goldman showed that the space of closed AdS $3$-manifolds is larger than originally expected \cite{G}, and Kulkarni and Raymond took up the problem of understanding all closed and geodesically complete AdS $3$-manifolds \cite{KR}. Among other things, they proved \cite[Theorem 5.2]{KR} that any torsion-free discrete group acting properly discontinuously on $\AdS$ is of the form 
\begin{equation}\label{form}
    \Gamma_{\rho_1,\rho_2} = \{(\rho_1(\gamma),\rho_2(\gamma)) : \gamma\in \Gamma\},
\end{equation}
where $\Gamma$ is the fundamental group of a (not necessarily compact) surface, and $\rho_1,\rho_2:\Gamma\to \PSL(2,\mathbb{R})$ are representations, with at least one of them Fuchsian (discrete and faithful). This is generalized for actions on rank $1$ Lie groups in \cite{K}.
\begin{remark}
Shortly after, Klingler \cite{Kl} proved that closed Lorentzian manifolds of constant curvature are geodesically complete. Thus, the completeness assumption can be dropped in the work of Kulkarni and Raymond on closed AdS $3$-manifolds.
\end{remark}
The natural next step is to understand which $\Gamma_{\rho_1,\rho_2}$ act properly discontinuously. In the cocompact case, Salein observed it is sufficient \cite{Sa2}, and Kassel proved it is necessary \cite{Ka2} that $\rho_1$ strictly dominates $\rho_2$ (defined below). Gu{\'e}ritaud and Kassel extended these results to surfaces with punctures and higher dimensional hyperbolic spaces \cite{GK}.
\begin{defn}
Let $\Gamma$ be a discrete group and let $\rho_k : \Gamma \to \textrm{Isom}(X_k,g_k)$ be representations into isometry groups of Riemannian manifolds $(X_k, g_k)$, $k=1,2$. $\rho_1$ (strictly) dominates $\rho_2$ if there is a $(\rho_1,\rho_2)$-equivariant $1$-Lipschitz ($\lambda$-Lipschitz, $\lambda<1$) map $f: X_1\to X_2$.
\end{defn}
A mapping as above is $(\rho_1,\rho_2)$-equivariant if for all $\gamma\in \Gamma$ and $x\in X_1$, $$f(\rho_1(\gamma) \cdot x) = \rho_2(\gamma)\cdot f(x).$$ When a group $\Gamma$ acts on a manifold with no reference to a representation $\rho_1$, we may just write $\rho_2$-equivariant. By the Selberg lemma, we only need to consider torsion-free groups.
\begin{thm}[Gu{\'e}ritaud-Kassel, Theorem 1.8 in \cite{GK}]\label{GK}
A finitely generated discrete group $\Gamma_{\rho_1,\rho_2}\subset \textrm{PSL}(2,\mathbb{R})\times \textrm{PSL}(2,\mathbb{R})$ of the form (\ref{form}) acts properly discontinuously and without torsion on $\textrm{AdS}^3$ if and only if $\rho_1$ is Fuchsian and strictly dominates $\rho_2$, up to interchanging $\rho_1$ and $\rho_2$.

The quotient is a Seifert-fibered AdS $3$-manifold over the hyperbolic surface $\mathbb{H}/\rho_1(\Gamma)$ such that the circle fibers are timelike geodesics.
\end{thm}

From the theorem above, it is clear that the $\textrm{AdS}$ $3$-manifold is compact if and only if $\rho_1$ is cocompact. Specific to closed surfaces, there was an open question: is every non-Fuchsian representation $\pi_1(S_g)\to \PSL(2,\mathbb{R})$ strictly dominated by a Fuchsian one? This question was answered by Deroin-Tholozan in \cite{DT} and Gu{\'e}ritaud-Kassel-Wolf in \cite{GKW}, using different methods. Deroin-Tholozan actually proved a more general result.
\begin{thm}[Deroin-Tholozan, Theorem A in \cite{DT}]\label{DT}
 Let $(X,\nu)$ be a $\textrm{CAT}(-1)$ Hadamard manifold with isometry group $G$ and $\rho: \pi_1(S_g)\to G$ a representation, $g\geq 2$. Then $\rho$ is strictly dominated by a Fuchsian representation, unless it stabilizes a totally geodesic copy of $\mathbb{H}$ on which the action is Fuchsian.
\end{thm}
Tholozan completed the story for closed $3$-manifolds in \cite{T}. Let $\mathcal{T}(S_g)$ be the Teichm{\"u}ller space of $S_g.$
\begin{thm}[Tholozan, Theorem 1 in \cite{T}]\label{T}
Fix $g\geq 2$ and a $\textrm{CAT}(-1)$ Hadamard manifold $(X,\nu)$ with isometry group $G$. The space of dominating pairs within $T(S_g)\times \textrm{Rep}^{nf}(\pi_1(S_g),G)$ is homeomorphic to $$\mathcal{T}(S_g)\times \textrm{Rep}^{nf}(\pi_1(S_g),G),$$ where $\textrm{Rep}^{nf}(\pi_1(S_g),G)$ is the space of representations that do not stabilize a totally geodesic copy of $\mathbb{H}$ on which the action is Fuchsian.
\end{thm}
The homeomorphism is fiberwise in the sense that for each $\rho \in \textrm{Rep}^{nf}(\pi_1(S_g),G)$, it restricts to a homeomorphism from $\mathcal{T}(S_g)\times \{\rho\}\to U\times \{\rho\}$, where $U\subset \mathcal{T}(S_g)$ is an open subset. The key point is that when $(X,\nu)=(\mathbb{H},\sigma)$, this is the deformation space of closed AdS quotients of $\PSL(2,\mathbb{R})$. The components of the deformation space are thus organized according to Euler numbers.

In \cite{AL}, Alessandrini-Li use results from \cite{DT} and \cite{T} to study AdS $3$-manifolds in the framework of non-abelian Hodge theory, finding an alternative construction of the closed AdS circle bundles over surfaces.
\end{subsection}

\begin{subsection}{Recent work}
There is now an interest in understanding non-compact quotients of $\PSL(2,\mathbb{R})$ (see \cite{DGK2} and \cite{DGK1},  for instance). For non-Fuchsian representations, we extended Theorem \ref{DT} in \cite[Theorem 1.2]{S}. Following the idea of Deroin-Tholozan, but using the infinite energy harmonic maps, we found that (with small exceptions) every non-Fuchsian representation is dominated by a Fuchsian one. Using strip deformations (see \cite{DGK1}) we then produced AdS $3$-manifolds \cite[Theorem 1.3]{S}. Some of the same results were proved by Gupta-Su in \cite{GS}, whose methods fall more in line with that of \cite{GKW}.

Near the end of the paper \cite{S}, we found something curious: pairs of representations $\rho_1,\rho_2$ that give rise to circle bundles with an anti-de Sitter structure that do not come from properly discontinuous actions on all of  $\AdS$. At the time we did not put much emphasis on the result; in fact, it's buried near the end of the paper as Proposition 7.7. The motivation for this paper is to explore representations such as $(\rho_1,\rho_2)$ in more depth. The representations all satisfy a geometric condition, which we call almost strict domination. Below, let $(X,\nu)$ be a Hadamard manifold with isometry group $G$.

\begin{defn}\label{asd}
 Let $\rho_1:\pi_1(S_{g,n})\to \PSL(2,\mathbb{R}), \rho_2: \pi_1(S_{g,n})\to G$ be two representations with $\rho_1$ Fuchsian. We say that $\rho_1$ almost strictly dominates $\rho_2$ if 
 \begin{enumerate}
     \item  for every peripheral $\zeta\in \Gamma$,  $\ell(\rho_1(\zeta))=\ell(\rho_2(\zeta))$, and
     \item there exists a $(\rho_1,\rho_2)$-equivariant $1$-Lipschitz map $g$ defined on the convex hull of the limit set of $\rho_1(\Gamma)$ in $\mathbb{H}$ such that the local Lipschitz constants are $<1$ inside the convex hull, and for peripherals $\zeta$ such that $\rho_1(\zeta)$ is hyperbolic, $g$ takes each boundary geodesic axis for $\rho_1(\zeta)$ isometrically to a geodesic axis for $\rho_2(\zeta)$.
 \end{enumerate}
\end{defn}
In the definition above, the global Lipschitz constant is $$\textrm{Lip}(g) = \sup_{y_1\neq y_2} \frac{d_\nu(g(y_1),g(y_2))}{d_\sigma(y_1,y_2)},$$ where $\sigma$ is the hyperbolic metric. The local one is $$\textrm{Lip}_x(g)=\inf_{r>0}\textrm{Lip}(g|_{B_r(x)}) = \inf_{r>0}\sup_{y_1\neq y_2 \in B_r(x)}\frac{d_\nu(g(y_1),g(y_2))}{d_{\sigma}(y_1,y_2)},$$ which by equivariance is a well-defined function on the convex core of  $\mathbb{H}/\rho_1(\pi_1(S_{g,n})$. In the language of \cite{GK}, the projection to $\mathbb{H}/\rho_1(\Gamma)$ of the stretch locus of an optimal Lipschitz map is exactly the boundary of the convex core. This property is very rare: it implies domination in the simple length spectrum (see \cite{GS}). In view of \cite{GK}, we will say a Lipschitz map $g:(\mathbb{H},\sigma)\to (X,\nu)$ is optimal if it satisfies the properties above.

To end this introductory portion, we list some other instances where domination has appeared. March{\'e} and Wolff use domination to answer a question of Bowditch and resolve the Goldman conjecture in genus $2$ \cite{MW}. Dai-Li study domination for higher rank Hitchin representations into $\PSL(n,\mathbb{C})$ \cite{DL}. It would be interesting to see if the almost strict condition generalizes meaningfully to higher rank.

\end{subsection}

\begin{subsection}{Main theorems: maximal surfaces}
We first fix some notations that we will keep throughout the paper. 
\begin{itemize}
    \item $\Sigma$ is surface with genus $g$ and $n$ punctures $p_1,\dots, p_n$, with $\chi(\Sigma)<0$. The deck group for the universal covering $\pi: \tilde{\Sigma}\to \Sigma$ is denoted by $\Gamma$.
    \item $\mathcal{T}(\Gamma)$ is the Teichm{\"u}ller space of classes of complete finite volume marked hyperbolic metrics on $\Sigma$.
    \item $\{\zeta_1,\dots, \zeta_n\}\subset \Gamma$ are the peripheral elements, i.e., those representing the simple closed curves enclosing $p_i$. If $n=1$, write $\zeta_1=\zeta$.
    \item $(X,\nu)$ is a $\textrm{CAT}(-1)$ Hadamard manifold with isometry group $G$.
    \item $(\mathbb{H},\sigma)$ denotes the hyperbolic space with constant curvature $-1$.
\end{itemize}
A representation $\rho:\Gamma\to G$ is irreducible if $\rho(\Gamma)$ fixes no point on the visual boundary $\partial_\infty X$. It is reductive if it is irreducible or if $\rho(\Gamma)$ preserves a geodesic in $X$.
\begin{thma}
Let  $\rho_1:\Gamma\to \PSL(2,\mathbb{R})$ and $\rho_2:\Gamma\to G$ be reductive representations with $\rho_1$ Fuchsian. $\rho_1$ almost strictly dominates $\rho_2$ if and only if there exists a complete finite volume hyperbolic metric $\mu$ on $\Sigma$ and a $\rho_1\times\rho_2$-equivariant tame spacelike maximal immersion from $$(\tilde{\Sigma},\tilde{\mu})\to (\mathbb{H}\times X,\sigma\oplus (-\nu)).$$ The maximal surfaces are not unique but are classified according to Proposition \ref{class}.
\end{thma}
We defer the definition of tame and the explanation of uniqueness to Section \ref{2.5} once we have more background on harmonic maps. Loosely, tameness means there is a constraint on the growth of the energy as we go into a puncture. Assuming this constraint, for each $\zeta_i$ such that $\rho(\zeta_i)$ is hyperbolic, there is a twist parameter $\theta_i\in \R$ that reflects how much the maximal immersion twists (the lift to $\tilde{\Sigma}$ of) a geodesic curve going up into the cusp. Assuming the map is tame, it is completely determined by its parameters at each peripheral. We also comment that almost strict domination is the same as strict domination when every $\rho_2(\zeta_i)$ is elliptic (see also \cite[Lemma 6.3]{S}).

Let's give some idea of the proof. For representations $\rho_1,\rho_2$, we define a functional $\E_{\rho_1,\rho_2}^\theta$ on the Teichm{\"u}ller space by $$\E_{\rho_1,\rho_2}^\theta(\mu) = \int_\Sigma e(\mu, h_\mu^\theta) - e(\mu, f_\mu^\theta) dA_\mu,$$ where $h_\mu^\theta$, $f_\mu^\theta$ are certain harmonic maps on $(\tilde{\Sigma},\tilde{\mu})$ that may have infinite energy, in the sense that $$\int_\Sigma e(\mu, h_\mu^\theta) dA_\mu= \int_\Sigma e(\mu, f_\mu^\theta) dA_\mu=\infty.$$ We show that this is always finite, provided the boundary lengths for $\rho_1$ and $\rho_2$ agree (Section 3). We then compute the derivative (Section 3), showing that critical points correspond to spacelike maximal surfaces (Proposition \ref{derivative}). To anyone working with harmonic maps, this is expected, but with no good theory of global analysis to treat infinite energy maps on surfaces with punctures, we have to work through some thorny details directly. In the course of our analysis, we develop a new energy minimization result (Lemma \ref{min}) that may be of independent interest.

Then we show that $\E_{\rho_1,\rho_2}^\theta$ is proper if and only if $\rho_1$ almost strictly dominates $\rho_2$. Here is an indication as to why this is true. Suppose we diverge along a sequence $(\mu_n)_{n=1}^\infty\subset \mathcal{T}(\Gamma)$ by pinching a simple closed curve $\alpha$. Then there is a collar around $\alpha$ whose length $\ell_n$ in $(\Sigma,\mu_n)$ is tending to $\infty$.  Almost strict domination implies $\ell(\rho_1(\alpha))>\ell(\rho_2(\gamma)$, and the analysis from \cite{S} shows that the total energy of the harmonic maps in the collar behaves like 

\begin{equation}\label{heur}
    \ell_n(\ell(\rho_1(\alpha))^2 - \ell(\rho_2(\alpha))^2) \to \infty.
\end{equation}
This reasoning, however, cannot be turned into a full proof. Two problems:
\begin{enumerate}
    \item Along a general sequence that leaves all compact subsets of $\mathcal{T}(\Gamma)$, the two harmonic maps could apriori behave quite differently in a thin collar. For instance we could have twisting in one harmonic map, which increases the energy, but no twisting in the other.
    \item For a general sequence, we also have little control over the energy outside of thin collars.
\end{enumerate}
We circumvent these issues as follows: if $g$ is an optimal map, then our energy minimization Lemma 3.12 implies that $$\E_{\rho_1,\rho_2}^\theta(\mu)\geq \int_\Sigma e(\mu, h_\mu^{\theta})-e(\mu, g\circ h_\mu^{\theta}) dA_\mu.$$ The integrand is positive, so now we can bound below by the energy in collars. The contracting property of $g$ then allows us to effectively study the energy in collars. In the end we make a rather technical geometric argument in order to find lower bounds similar to (\ref{heur}) along diverging sequences.

We also remark that even in the non-compact but finite energy setting, the result on the derivative of the energy functional was not previously contained in the literature. Hence we record it below. 

\begin{prop}\label{derivativefin}
Let $\rho:\Gamma\to G$ be a reductive representation with no hyperbolic monodromy around cusps, so that equivariant harmonic maps have finite energy. Then the energy functional $E_\rho:\mathcal{T}(\Gamma)\to [0,\infty)$ that records the total energy of a $\rho$-equivariant harmonic map from $(\tilde{\Sigma},\mu)\to (X,\nu)$ is differentiable, with derivative at a hyperbolic metric $\mu$ given by $$dE_\rho[\mu](\psi) = -4\Real \langle \Phi,\psi\rangle.$$ Here $\Phi$ is the Hopf differential of the harmonic map at $\mu$.
\end{prop}
See Sections 2 and 3 for notations. The proof can actually be extended to non-positively curved settings (see Remark \ref{npc}).

\end{subsection}
\begin{subsection}{Main theorems: parametrizations}
The next theorem concerns the space of almost strictly dominating pairs. We denote by $\textrm{Hom}^*(\Gamma,G)\subset \textrm{Hom}(\Gamma,G)$ the space of reductive representations. $G$ acts on $\textrm{Hom}^*(\Gamma,G)$ by conjugation, and we define the representation space as $$\textrm{Rep}(\Gamma, G)=\textrm{Hom}^*(\Gamma,G)/G.$$ 
In general this may not be a manifold, but it can have nice structure depending on $G$. For surfaces with punctures we would like to prescribe behaviour at the punctures. 
\begin{defn}
Fix a collection of conjugacy classes $\mathbf{c}=(c_i)_{i=1}^n$ of elements in $G$. The relative representation space $\textrm{Rep}_{\mathfrak{c}}(\Gamma,G)$ is the space of reductive representations taking $\zeta_i$ into $c_i$, modulo conjugation.
\end{defn}
We require one technical assumption on the group $G$: that if we choose a good covering of $\Sigma$ and let $\chi(\Gamma,G)$ denote the space of $G$-local systems with respect to this covering that have reductive holonomy, then the projection from $\chi(\Gamma,G)\to \textrm{Rep}(\Gamma,G)$ is a locally trivial principal bundle. We demand the same for the relative representation space, instead considering local systems whose holonomy representations respect $\mathbf{c}$. This assumption is satisfied under most cases of interest in Higher Teichm{\"u}ller theory, for instance if $G$ is a linear algebraic group (see \cite[Chapter 5]{La}).

 Within $\textrm{Rep}_{\mathbf{c}}(\Gamma,G)$, we have the subset $\textrm{Rep}_{\mathbf{c}}^{nf}(\Gamma,G)$ of representations that do not stabilize a plane of constant curvature $-1$ on which the action is Fuchsian. The almost strict domination condition is invariant under conjugation for both representations, so we can define $\textrm{ASD}_{\mathbf{c}}(\Gamma,G)$ to be the subspace of pairs of representations $\rho_1:\Gamma\to\PSL(2,\mathbb{R})$, $\rho_2:\Gamma\to G$ such that $\rho_1$ is Fuchsian and almost strictly dominates $\rho_2$. Necessarily, $\rho_1$ lies in the Teichm{\"u}ller space $\mathcal{T}_{\mathbf{c}}(\Gamma)$ (defined in Section \ref{2.2}).

 Assume there are $m$ peripherals such that $c_1,\dots,c_m$ are hyperbolic conjugacy classes. In Section 5.1, we explain that once we fix twist parameters $(\theta_1,\dots,\theta_m)\in \mathbb{R}^m$, there is a procedure that takes as input a point in $\mathcal{T}(\Gamma)\times \textrm{Rep}^{nf}_{\mathbf{c}}(\Gamma,G)$ and produces a tame spacelike maximal immersion with those parameters, and thus, by Theorem A, an almost strictly dominating pair. Formalizing the procedure, we obtain a family of maps from $\mathcal{T}(\Gamma)\times \textrm{Rep}^{nf}_{\mathbf{c}}(\Gamma,G)\to \textrm{ASD}_{\mathbf{c}}(\Gamma,G)$, which we prove are homeomorphisms.
\begin{thmb}
 Assume as above that there are $m$ peripherals giving hyperbolic conjugacy classes. For each choice of parameters $\theta=(\theta_1,\dots,\theta_m)\in \mathbb{R}^m$, there exists a homeomorphism $$\Psi^\theta: \mathcal{T}(\Gamma)\times \textrm{Rep}^{nf}_{\mathbf{c}}(\Gamma,G)\to \textrm{ASD}_{\mathbf{c}}(\Gamma,G).$$ Moreover, the homeomorphism is fiberwise in the sense that for each $\rho \in \textrm{Rep}_{\mathbf{c}}^{nf}(\Gamma,G)$, it restricts to a homeomorphism $$\Psi_\rho^\theta:\mathcal{T}(\Gamma)\times \{\rho\} \to U\times \{\rho\} \subset ASD_{\mathbf{c}}(\Gamma,G),$$ where $U$ is a non-empty open subset of the Teichm{\"u}ller space $\mathcal{T}_{\mathbf{c}}(\Gamma)$.
\end{thmb}
 The mappings $\Psi^\theta$ are defined in essentially the same way as the map $\Psi$ from \cite{T}. Theorem B should be compared with Theorem \ref{T}. 
\end{subsection}

\begin{subsection}{Main theorems: AdS $3$-manifolds}
Concerning AdS $3$-manifolds, the following explains the relationship with spacelike immersions. 
\begin{prop}\label{equiv}
Given a $\rho_1$-invariant domain $V\subset \mathbb{H}$ on which $\rho_1$ acts properly discontinuously, there is a bijection between 
\begin{enumerate}
    \item $(\rho_1,\rho_2)$-equivariant maps $g: V\to \mathbb{H}$ that are locally strictly contracting, i.e., $$d_\sigma(g(x),g(y))<d_\sigma(x,y)$$ for $x\neq y$,
    \item  and circle bundles $p:\Omega/(\rho_1\times \rho_2(\Gamma))\to V$, where $\Omega\subset \AdS$ is a domain on which $\rho_1\times \rho_2$ acts properly discontinuously and such that each circle fiber lifts to a complete timelike geodesics in $\AdS$.
\end{enumerate}
\end{prop}
Indeed, given a spacelike maximal surface $(h,f)$ defined on $V\subset \mathbb{H}$, we will see later on that $g=h\circ f^{-1}$ is locally strictly contracting on a domain. The implication from (1) to (2) is a slight generalization of the work of Gu{\'e}ritaud-Kassel in \cite{GK}, and should be known to experts.
\begin{remark}
The proof in \cite{KR} that properly discontinuous subgroups of $\AdS$ are of the form $\Gamma_{\rho_1,\rho_2}$ rests on their main lemma that there is no $\mathbb{Z}^2$-subgroup acting properly discontinuously. The proof is local, and one can adapt to show that any torsion-free discrete group acting properly discontinuously on a domain in $\AdS$ takes this form.
\end{remark}
\begin{remark}
A version of this holds more generally for geometric structures modelled on some rank $1$ Lie groups. See Section 6.2.
\end{remark}

Specializing to almost strict domination, we have the following.
\begin{thmc}
Let $\rho_1,\rho_2:\Gamma\to \PSL(2,\mathbb{R})$ be two reductive representations with $\rho_1$ Fuchsian. The following are equivalent.
\begin{enumerate}
\item $\rho_1$ almost strictly dominates $\rho_2$.
\item $\rho_1\times\rho_2$ acts properly discontinuously on a domain $\Omega\subset \AdS$ and induces a fibration from $\Omega/(\rho_1\times \rho_2(\Gamma))$ onto the interior of the convex core of $\mathbb{H}/\rho_1(\Gamma)$ such that each fiber is a timelike geodesic circle. Moreover, when there is at least one peripheral $\zeta$ with $\rho_1(\zeta)$ hyperbolic, no such domain in $\AdS$ can be continued to give a fibration over a neighbourhood of the convex core.
    \item There exists a complete hyperbolic metric $\mu$ on $\Sigma$ and a $(\rho_1,\rho_2)$-equivariant embedded tame maximal spacelike immersion from $(\tilde{\Sigma},\tilde{\mu})\to (\mathbb{H}\times\mathbb{H},\sigma\oplus (-\sigma))$.
\end{enumerate}
Fixing a collection of conjugacy classes $\mathbf{c}$, there is a fiberwise homeomorphism $$\Psi : \mathcal{T}_{\mathbf{c}}(\Gamma)\times \textrm{Rep}_{\mathbf{c}}^{nf}(\Gamma,\PSL(2,\mathbb{R}))\to \textrm{ASD}_{\mathbf{c}}(\Gamma,\PSL(2,\mathbb{R})).$$ If we restrict the domain to classes of irreducible representations, the image identifies with a continuously varying family of AdS $3$-manifolds.
\end{thmc}
Components of the space of AdS $3$-manifolds are classified by the relative Euler numbers (see \cite{BIW}). The only piece that doesn't follow quickly from Theorems A, B, and Proposition \ref{equiv} is the implication from (2) to (1). To prove this part, we draw on the work of Gu{\'e}ritaud-Kassel on maximally stretched laminations \cite{GK} and show that the stretch locus (Definition 6.4) of an optimally Lipschitz map is exactly the boundary of the convex hull of the limit set.

In (2), the domain $\Omega$ is all of $\AdS$ if and only if every $\rho_2(\zeta_i)$ is elliptic. This is a consequence of \cite[Lemma 2.7]{GK} and the properness criteria, Theorem \ref{GK}. Also related to (2), one can get incomplete AdS $3$-manifolds fibering over larger subsurfaces, but still not extending to the whole surface, by doing strip deformations (see Section 6.3). 

\end{subsection}

\begin{subsection}{Future Work}
There are many directions to go from here. For now we point out two.
\begin{enumerate}
    \item In Definition \ref{asd}, we assume that $\rho_1$ is Fuchsian. We can relax this, and allow for $\rho_1$ to be the holonomy of a hyperbolic conifold. One can find maximal spacelike immersions, again with applications to anti-de Sitter geometry. We plan to address this in future work.
    \item In a similar vein, one can remove the tame hypothesis (Definition 2.7). Gupta builds harmonic diffeomorphisms to crowned hyperbolic surfaces whose Hopf differentials have poles of order $\geq 3$ at cusps \cite{Gu}. In forthcoming work, among other things, we plan to introduce a class of AdS $3$-manifolds that fiber over crowned surfaces. 
\end{enumerate}
\end{subsection}

\begin{subsection}{Outline}
\begin{itemize}
    \item In Section 2 we give the background content for Theorem A and Theorem B: spaces of representations, harmonic maps, and maximal surfaces. We also set up the proof of Theorem A by defining the energy functionals $\E_{\rho_1,\rho_2}^\theta$.
    \item  In Section 3, we treat the functional $\E_{\rho_1,\rho_2}^\theta$: well-definedness and the derivative.
    \item In Section 4 we show that $\E_{\rho_1,\rho_2}^\theta$ is proper if and only if $\rho_1$ almost strictly dominates $\rho_2$ (see 1.3 for the proof idea).
    \item We prove Theorem B in Section 5 by studying variations of minimizers of $\E_{\rho_1,\rho_2}^\theta$ (similar to \cite[Section 2]{T}).
    \item Section 6 is a change of pace. After giving an overview of the relevant aspects of AdS geometry, we prove Proposition \ref{equiv} and Theorem C. We close the paper with some comments on the AdS geometry, as well as Higgs bundles.
\end{itemize}
\end{subsection}

\begin{subsection}{Acknowledgements}
I'd like to thank Qiongling Li, Vlad Markovic, and Nicolas Tholozan for helpful discussions and correspondence. Thanks also to an anonymous referee for comments on an earlier draft. I am funded by the FNR grant O20/14766753, \it{Convex Surfaces in Hyperbolic Geometry.}
\end{subsection}

\end{section}
\begin{section}{Preliminaries}

\begin{subsection}{Teichm{\"u}ller spaces}\label{2.2} 
Firstly, we review isometries for Hadamard manifolds. The translation length of an isometry $\gamma\in G$ is $$\ell(\gamma) = \inf_{x\in X}d_\nu(\gamma\cdot x,x),$$ and $\gamma$ is 
\begin{itemize}
    \item elliptic if $\ell(\gamma)=0$ and the infimum is attained,
    \item parabolic if $\ell(\gamma)=0$ and the infimum is not attained, and
    \item hyperbolic if $\ell(\gamma)>0$.
\end{itemize}
Elliptic isometries fix points inside $X$, parabolic isometries fix points on $\partial_\infty X$, and a hyperbolic isometry stabilizes a unique geodesic on which it acts by translation of length $\ell(\gamma)$. Around a peripheral $\zeta$, we say a representation $\rho$ has hyperbolic, parabolic, or elliptic mondromy if $\rho(\zeta)$ is hyperbolic, parabolic, or elliptic respectively.

Fix a collection of conjugacy classes $\mathbf{c}=(c_i)_{i=1}^n\subset \PSL(2,\mathbb{R})$.
\begin{defn}
Suppose $\mathbf{c}$ has no elliptic isometries. We refer to the Teichm{\"u}ller space $\mathcal{T}_{\mathbf{c}}(\Gamma)$ for $\mathbf{c}$ as one of the components of $\textrm{Rep}_\mathfrak{c}(\Gamma,\PSL(2,\mathbb{R}))$ consisting of discrete and faithful representations (there are two such components). Representations in Teichm{\"u}ller space are called Fuchsian.
\end{defn}
Throughout the paper, by ``hyperbolic metric on $\Sigma$" we mean a complete finite volume hyperbolic metric, unless specified otherwise. We say two such metrices $\mu_1,\mu_2$ are equivalent if there exists an orientation preserving diffeomorphism $\psi$, isotopic to the identity, such that $\psi^*\mu_1 = \mu_2$. A Teichm{\"u}ller space agrees with the space of classes $[\mu]$ of hyperbolic metrics on $\Sigma$ with cusps and geodesic boundary components whose lengths are prescribed by $\mathbf{c}$. When we make no reference to the boundary data, we refer to the Teichm{\"u}ller space $\mathcal{T}(\Gamma)$ such that all peripheral monodromy is parabolic. For most of the paper, when we write $\mathcal{T}(\Gamma)$, we are using the perspective of hyperbolic metrics, and with $\mathcal{T}_{\mathbf{c}}(\Gamma)$ we are considering representations.
 
For the sake of convenience, we abuse notation slightly:  if $\mathbf{c}=(c_i)_{i=1}^n$ has elliptic isometries, we write $\mathcal{T}_{\mathbf{c}}(\Gamma)$ to be the Teichm{\"u}ller space with boundary length $\ell(c_i)$ for each $c_i$ with positive length, and cusps for each $c_i$ that is elliptic or parabolic.

We need some terminology for hyperbolic surfaces. Let $\mu$ be any metric on $\Sigma$, with no constraint on the volume. A cusp region is a neighbourhood surrounding a puncture of $\Sigma$ that identifies isometrically with 
\begin{equation}\label{17}
    U(\tau):=\{z=x+iy: (x,y)\in [0,\tau]\times [a,\infty)\}/\langle z\mapsto z+\tau\rangle,
\end{equation}
equipped with the hyperbolic metric $y^{-2}|dz|^2$. 
 \begin{defn}
  The convex core $C(\Sigma,\mu)$ of $(\Sigma,\mu)$ is the quotient of the convex hull of the limit set of $\Gamma$ by the action of $\Gamma$. We denote the convex hull of the limit set by $\tilde{C}(\Sigma,\mu)$. For quotients $\mathbb{H}/\rho_1(\Gamma)$, we omit the metric from the notation.  
 \end{defn}
It is the minimal convex set such that the inclusion $C(\Sigma,\mu)\to \Sigma$ is a homotopy equivalence. The convex core is a finite volume hyperbolic surface with finitely many cusps and geodesic boundary components. $(\Sigma,\mu)$ can be recovered from $C(\Sigma,\mu)$ by attaching infinite funnels along the boundary components. In the language of representations, the monodromy of the holonomy representation around a cusp is parabolic, and for a geodesic boundary component it is hyperbolic.
\end{subsection}

\begin{subsection}{Harmonic maps}
Let $\rho:\Gamma\to G$ be a representation and $\mu$ a metric on $\Sigma$. Given a $C^2$ $\rho$-equivariant map $f: (\tilde{\Sigma},\tilde{\mu})\to (X,\nu)$, the derivative $df$ defines a section of the endomorphism bundle $T^*\tilde{\Sigma}\otimes f^*TX$. The Levi-Civita connections on $(\tilde{\Sigma},\tilde{\mu})$ and $(X,\nu)$ induce a connection on this bundle, and we define the tension field $\tau(\mu,\nu,f)\in \Gamma(f^*TX)$ by $$\tau(\mu,\nu,f) = \textrm{trace}_\mu \nabla df.$$
\begin{defn}
A $\rho$-equivariant map $f:(\tilde{\Sigma},\tilde{\mu})\to (X,\nu)$ is harmonic if $\tau(\mu,\nu,f)=0$.
\end{defn}
We assume throughout that our metrics are smooth, which implies harmonic maps are smooth. Existence of equivariant harmonic maps for irreducible representations of closed surface groups to $\PSL(2,\mathbb{C})$ was proved by Donaldson \cite{D}, who used the heat flow developed by Eells-Sampson \cite{EL}. Corlette and Labourie extended to more general settings \cite{C} \cite{L}.

We simultaneously view $(\Sigma,\mu)$ as a Riemann surface with the holomorphic structure compatible with $\mu$. In a local holomorphic coordinate $z$ on $(\tilde{\Sigma},\tilde{\mu})$ in which $\tilde{\mu} = \mu(z)|dz|^2$, the pullback metric $f^*\nu$ decomposes as
\begin{equation}\label{2}
    f^*\nu = e(\mu,f)\mu(z)dzd\overline{z} + \Phi(f)(z)dz^2 + \overline{\Phi}(f)(z)d\overline{z}^2.
\end{equation}
Since $G$ is acting by isometries and $f$ is equivariant, $f^*\nu$ descends to a symmetric $2$-form on $\Sigma$. We call this the pullback metric, even though it may have degenerate points. $$e(\mu,f) = \frac{1}{2}\textrm{trace}_{\tilde{\mu}} f^*\nu$$ is a well-defined function on $\Sigma$ called the energy density with respect to $\mu$, and $\Phi(f)(z)$ is a holomorphic quadratic differential on $\Sigma$---a holomorphic section of the symmetric square of the canonical bundle---called the Hopf differential.

The energy density can be defined for any equivariant map, harmonic or not. When the base surface is closed, harmonic maps are critical points for the energy functional $$f\mapsto \int_\Sigma e(\mu,f) dA_\mu$$ (defined on a suitable Sobolev space), where $dA_\mu$ is the area form. This ``total energy" is conformally invariant and depends only on the class of $\mu$ in Teichm{\"u}ller space.
\end{subsection}

\begin{subsection}{Infinite energy harmonic maps} For surfaces with punctures, existence is proved for different families of harmonic maps from punctured surfaces in a few places, for example, \cite{Gu}, \cite{KM}, \cite{Lo}, and \cite{W}. In \cite{S}, we proved the most general result for equivariant harmonic maps to negatively curved Hadamard manifolds and studied their properties in detail. The results of \cite{Lo} and \cite{W} now correspond to the case of Fuchsian representations. This is summarized in our result below. Set $\Lambda(\theta)=(1-\theta^2) -i2\theta$. Reorganize the $p_i$ so that $\rho(\zeta_i)$ is hyperbolic for $i=1,\dots, m$. For each such $p_i$, pick a $\theta_i\in\mathbb{R}$. The vector $\theta=(\theta_1,\dots, \theta_m)$ is called a twist parameter.
\begin{thm}[Theorem 1.1 in \cite{S}]\label{S}
Let $\rho:\Gamma\to G$ be a reductive representation. Then there exists a $\rho$-equivariant harmonic map $f^\theta:(\tilde{\Sigma},\tilde{\mu})\to (X,\nu)$ such that if $\zeta$ is a peripheral isometry and $\theta\in \mathbb{R}$, the Hopf differential $\Phi$ has the following behaviour at the corresponding cusp
\begin{itemize}
    \item if $\rho(\zeta)$ is parabolic or elliptic, $\Phi$ has a pole of order at most $1$ and
    \item if $\rho(\zeta)$ is hyperbolic, $\Phi$ has a pole of order $2$ with residue $$-\Lambda(\theta_i)\ell(\rho(\zeta))^2/16\pi^2.$$
\end{itemize}
If $\rho$ does not fix a point on $\partial_\infty X$, then $f^\theta$ is the unique harmonic map with these properties. If $\rho$ stabilizes a geodesic, then any other harmonic map with the same asymptotic behaviour differs by a translation along that geodesic.
\end{thm}
The residue of $\Phi$ is the coefficient attached to the $z^{-2}$ term of the Laurent expansion in a holomorphic coordinate taking the cusp to $\mathbb{D}^*$. This is independent of the coordinate.

The theorem says that a harmonic map is determined by the choice of twist parameter, with an exception if $\rho$ is reducible. To elaborate on the nomenclature, take $\zeta_i$ with $\rho(\zeta_i)$ hyperbolic, and choose a constant speed parametrization $\alpha_i$ for the geodesic axis $\beta_i$ of $\rho(\zeta_i)$. Let $\mathcal{C}$ be the cylinder
\begin{equation}\label{9}
    \mathcal{C}=\{(x,y)\in[0,1]^2\}/\langle (0,y)\sim (1,y)\rangle.
\end{equation}
with the flat metric. There is a ``model mapping" $\tilde{\alpha}_i^{\theta_i}:\mathcal{C}\to \beta_i$ defined as follows. For $\theta_i=0$, we set $$\tilde{\alpha}_i(x,y)=\alpha_i(x),$$ a constant speed projection onto the geodesic. For $\theta_i\neq 0$, $\tilde{\alpha}_i^\theta$ is defined by precomposing $\tilde{\alpha}_i$ with the fractional Dehn twist of angle $\theta_i$, the map given in coordinates by 
\begin{equation}\label{16}
    (x,y)\mapsto (x+\theta_i y, y).
\end{equation}
Choosing a cusp neighbourhood $U$ of a $p_i$, we take conformal maps $i_r:\mathcal{C}\to U$ that take the boundaries linearly to $\{(x,y):y=r\}$ and $\{(x,y):y=r+1\}$. The mappings $f^\theta \circ i_r:\mathcal{C}\to (X,\nu)$ are harmonic, and as $r\to \infty$ they converge in the $C^\infty$ sense to a harmonic mapping that differs from $\tilde{\alpha}_i^{\theta_i}:\mathcal{C}\to \beta_i$ by a constant speed translation along the geodesic (see \cite[Section 5]{S}).

When the representation is Fuchsian, the harmonic map $h^\theta$ descends to a diffeomorphism from $\mathbb{H}\to C(\mathbb{H}/\rho(\Gamma))$ (see \cite{W}).
\end{subsection}
\begin{subsection}{Spacelike maximal surfaces}\label{2.5}
Let $\rho_1:\Gamma\to \PSL(2,\mathbb{R})$, $\rho_2:\Gamma\to G$ be reductive representations. Then $\rho_1\times \rho_2: \Gamma\to \PSL(2,\mathbb{R})\times G$ defines its own representation.
\begin{defn}
A $\rho_1\times\rho_2$-equivariant map $F:(\tilde{\Sigma},\tilde{\mu})\to (\mathbb{H}\times X, \sigma\oplus (-\nu))$ is maximal if the image surface has zero mean curvature. It is spacelike if the pullback metric $F^*(\sigma\oplus (-\nu))$ is non-degenerate and Riemannian.
\end{defn}
The vanishing of the mean curvature is equivalent to the condition that $F$ is harmonic and conformal. Using the product structure, we can write $$F=(h,f),$$ where $h,f$ are $\rho_1,\rho_2$-equivariant harmonic maps, and from (\ref{2}), $$\Phi(F)=\Phi(h)-\Phi(f).$$ Since $F$ is conformal, $\Phi(h)$ and $\Phi(f)$ agree.  
\begin{defn}
A spacelike maximal surface $F:(\tilde{\Sigma},\tilde{\mu})\to (\mathbb{H}\times X, \sigma\oplus (-\nu))$ is called tame if the Hopf differentials of the harmonic maps have poles of order at most $2$ at the cusps.
\end{defn}
At this point, we can readily prove the easy direction of Theorem A.
\begin{lem}
If $\rho$ is Fuchsian, the existence of a tame spacelike maximal surface implies almost strict domination.
\end{lem}
\begin{proof}
Let $F$ be such a maximal surface and split it as $F=(h,f)$. Note that $\Phi(h)=\Phi(f)$ implies $\ell(\rho_1(\zeta_i))=\ell(\rho_2(\zeta_i))$ for all $i$. Indeed, $\ell(\rho_k(\zeta_i))=0$ if and only if the Hopf differential has a pole of order at most $1$ at the cusp. And if $\ell(\rho_k(\zeta_i))>0$, this is because the residue at each cusp is determined entirely by the choice of twist parameter and the translation length $\ell(\rho_k(\zeta_i))$.

From \cite{W}, the harmonic map $h$ is a diffeomorphism onto the interior of the convex hull of the limit set of $\rho_1(\Gamma)$. Moreover, the mapping $f\circ h^{-1}$ is well defined. Since $F$ is spacelike (actually, by \cite[Proposition 3.13]{S}, maximal implies spacelike here), $f\circ h^{-1}$ is locally strictly contracting on the interior of the convex hull of the limit set for the holonomy representation associated to $h^*\sigma$. Since the harmonic maps both converge in a suitable sense to a projection onto the geodesic axis as in Section 2.4, it has local Lipschitz constant exactly $1$ on the boundary geodesic. By definition, $\rho_1$ almost strictly dominates $\rho_2$.
\end{proof}
 \begin{remark}
Note that $\rho_2$ cannot be Fuchsian, by an application of Gauss-Bonnet.
\end{remark}
The proof of the other direction of Theorem A will go as follows. Take a pair $(\rho_1,\rho_2)$ such that $\rho_1$ almost strictly dominates $\rho_2$. We want to show there is a hyperbolic metric $\mu$ on $\Sigma$ such that, after choosing twist parameters in some way, the associated $(\rho_1,\rho_2)$-equivariant surface from $(\tilde{\Sigma},\tilde{\mu})\to (\mathbb{H}\times X, \sigma \oplus (-\nu))$ is spacelike and maximal. Clearly, we must have the same twist parameter $\theta$ for both $\rho_1$ and $\rho_2$, and it turns out we can choose any $\theta$, as we will now explain. Given a metric $\mu$, let $h_\mu^\theta,f_\mu^\theta$ be harmonic maps for $\rho_1,\rho_2$ respectively with twist parameter $\theta$. Define $$\E_{\rho_1,\rho_2}^\theta: \mathcal{T}(\Gamma) \to \mathbb{R}$$ by $$\E_{\rho_1,\rho_2}^\theta(\mu) = \int_\Sigma e(\mu, h_\mu^\theta) - e(\mu, f_\mu^\theta) dA_\mu.$$ This does not depend on the metric $\mu\in [\mu]$, by conformal invariance (Section 3.1). When $\rho_2$ is reducible, harmonic maps are not unique, but the energy density does not depend on the choice of harmonic map, so the integral is well defined. A priori, it is not given that the integral defining $\E_{\rho_1,\rho_2}^\theta(\mu)$ is finite, for the harmonic maps themselves could have infinite energy. Finiteness will be proved in Proposition \ref{welldef}. The main goal of Section 3 is to prove Proposition \ref{derivative}. 
\begin{prop}\label{derivative}
Given $(\rho_1,\rho_2)$ with $\ell(\rho_1(\zeta_i))=\ell(\rho_2(\zeta_i))$ for all $i$, the functional $\E_{\rho_1,\rho_2}^\theta:\mathcal{T}(\Gamma)\to \mathbb{R}$ is differentiable with derivative at a hyperbolic metric $\mu$ given by $$d\E_{\rho_1,\rho_2}^\theta[\mu](\psi) = -4\Real \langle \Phi(h_\mu^\theta)-\Phi(f_\mu^\theta),\psi\rangle.$$
\end{prop}
Thus, critical points of $\E_{\rho_1,\rho_2}^\theta$ correspond to maximal surfaces. And by \cite[Proposition 3.13]{S}, any critical point is a spacelike immersion. In Section 4, we prove that the almost strict domination hypothesis implies there is a metric $\mu$ that minimizes $\E_{\rho_1,\rho_2}^\theta$, and that this is the only critical point. Granting this, we can discuss uniqueness.
\begin{prop}\label{class}
Suppose $\rho_1$ almost strictly dominates $\rho_2$ and that $\rho_2$ is irreducible. Assume that for every choice of twist parameter $\theta$, each $\E_{\rho_1,\rho_2}^\theta$ admits a unique critical point. Then every spacelike maximal immersion is of the form $(h_{\mu_\theta}^\theta,f_{\mu_\theta}^\theta)$, where $\mu_\theta$ is the minimizer for $\E_{\rho_1,\rho_2}^\theta$. If $\rho_2$ is reducible, then for each $\theta$ there is a a $1$-parameter family of tame maximal surfaces, and each one is found by translating $f_\mu^{\theta}$ along a geodesic axis.
\end{prop}
\begin{proof}
In the irreducible case, let $F=(h',f'): (\tilde{\Sigma},\tilde{\mu})\to (X,\nu)$ be a spacelike maximal immersion. By the uniqueness statement in Theorem \ref{S}, $h'$ must be of the form $h_\mu^\theta$ for some parameter $\theta$. Thus $f=f_\mu^\theta$, and by our assumption, $\mu$ is the unique critical point for $\E_{\rho_1,\rho_2}^\theta$. 

In the reducible case, all $\rho_2(\zeta_i)$ that are hyperbolic must translate along the same geodesic. Since translating along the geodesic does not change the energy density or the Hopf differential, the one minimizer for $\E_{\rho_1,\rho_2}^\theta$ is the only one that yields maximal surfaces.
\end{proof}

\end{subsection}
\end{section}

\begin{section}{The derivative of the energy functional}
In this section we prove Proposition \ref{derivative}. We assume $\Sigma$ has only one puncture $p$ with peripheral $\zeta$---the analysis is local so the general case is essentially the same. When $\rho_1(\zeta),\rho_2(\zeta)$ are hyperbolic, we assume the twist parameter $\theta$ is $0$, and we also write $\ell$ in place of $\ell(\rho_k(\zeta))$. In Section \ref{3.9}, we explain the adjustments for the general case. With these assumptions, it causes no harm to write $\E$ for $\E_{\rho_1,\rho_2}^\theta$. Not only in this section but for the rest of the paper, we write $h_\mu,f_\mu$ for $h_\mu^\theta, f_\mu^\theta$ when $\theta$ is zero. 

\begin{subsection}{Analytic preliminaries}
Let $f:(\tilde{\Sigma},\tilde{\mu})\to (X,\nu)$ be harmonic. According to the splitting of the complexified tangent bundle $T\tilde{\Sigma}\otimes \mathbb{C}$ into $(1,0)$ and $(0,1)$ components, the energy density decomposes as $$e(\mu,f) = H(\mu,f) + L(\mu,f),$$ the holomorphic energy and the anti-holomorphic energy. When $(X,\nu)$ is a surface with compatible complex structure, we can define these quantities in the usual way as norms of holomorphic and anti-holomorphic derivatives. In general, we set $H(\mu,f)\geq L(\mu,f)$ to be the unique solutions of the system $x+y=e(\mu,f)$, $xy=\mu^{-2}|\Phi|^2,$ such that $x\geq y$, which are well-defined since $e(\mu,f) - 2\mu^{-2}|\Phi|^2\geq 0$. When the context is clear, we write $H=H(\mu,f)$, and the same for other analytic quantities. We collect some notations and relations that, in the sequel, we use without comment.
\begin{itemize}
\item When unspecified, $C$ denotes a constant that may grow in the course of a proof. 
\item Setting $x=x_1,y=x_2$, let $e_{\alpha\beta}(f)=f^*\nu\Big(\frac{\partial f}{\partial x_\alpha}, \frac{\partial f}{\partial x_\beta}\Big)$, so that $e(\mu,f)=\frac{1}{2}\mu^{\alpha\beta}e_{\alpha\beta}(f)$.
    \item The Jacobian $J=J(\mu,f)$ satisfies $J=H-L$.
    \item For a conformal metric $\mu=\mu(z)|dz|^2$, $||\Phi||^2:=\mu^{-2}|\Phi|^2=HL$.
    \item Set $|\psi| = L^{1/2}/H^{1/2}$. If $(X,\nu)=(\mathbb{H},\sigma)$ and $\rho$ is Fuchsian, $\psi = \frac{f_{\overline{z}}}{f_z}$ is the Beltrami form.
    \end{itemize}
We will make use of the following results.
\begin{itemize}
    \item Total energy is conformally invariant: if $i:(\Omega,g)\to (\tilde{\Sigma},\tilde{\mu})$ is conformal, then $f\circ i$ is harmonic.
    \item If $f_1,f_2$ are both harmonic, then the distance function $p\mapsto d_\nu(f_1(p),f_2(p))$ descends to a subharmonic function on $(\Sigma,\mu)$.
    \item Cheng's lemma \cite{Ch}: let $B_{\tilde{\mu}}(z,r)\subset (\tilde{\Sigma},\tilde{\mu})$ be a metric ball and $f:B_{\tilde{\mu}}(z,r)\to (X,\nu)$ a harmonic map whose image lies in a $\nu$-ball of radius $R_0$. Then $$e(\mu,f)\leq R_0^2 \cdot \frac{(1+r)^2}{r^2}.$$
    \item Proposition 3.13 from \cite{S}: if $h$ is a Fuchsian harmonic map with same Hopf differential as $f$ (this exists by Theorem 1.4 of \cite{S}), then $H(\mu,f), L(\mu,f)\leq H(\mu,h)$.
\end{itemize}
We prepare notation for dealing with cusps. On a hyperbolic surface $(\Sigma,\mu)$, there is a cusp region $U(\tau):=\{z=x+iy: (x,y)\in [0,\tau]\times [2,\infty)\}/\langle z\mapsto z+\tau\rangle,$ where $\tau=\tau_\mu$ depends on $\mu$. For $r\geq 2$, we define $(\Sigma_r,\mu)$ to be $\Sigma\backslash \{z\in U(\tau): y>r\}$, and we put $(\tilde{\Sigma}_r,\tilde{\mu})$ to be the preimage in $\tilde{\Sigma}$. Note that, as a set of points, this depends on $\mu$. As in Section 2.4, $\mathcal{C}$ is a conformal cylinder, and when we say a ``conformal cylinder for $\mu$," we mean the length is adjusted to be $\tau$ and the height is $1$, so that there are conformal maps $\mathcal{C}\to U(\tau)$. We define the maps $$i_r:\mathcal{C}\to \Sigma_{r+1}\backslash\Sigma_r \subset \Sigma$$ by $(x,y)\mapsto x + i(y+r)$, which are used to study asymptotics of harmonic maps.

It is often helpful to perturb the metric in the cusp. A metric $\mu$ is expressed in the coordinate of $U(\tau)$ as $\mu(z) = y^{-2}|dz|^2$. The flat-cylinder metric $\mu^f$ is defined by $\mu$ in $\Sigma_2=\Sigma\backslash U$, the flat metric $|dz|^2$ in the cusp coordinates on $\Sigma\backslash \Sigma_3$, and smoothly interpolated in between. This is conformally equivalent to $\mu$, so harmonic maps for $\mu$ and $\mu^f$ are the same. 

We now recall the construction of infinite energy harmonic maps from \cite[Section 5]{S} for a reductive representation $\rho$ such that $\rho(\zeta)$ is hyperbolic with $\theta=0$. Denote by $\beta$ the geodesic axis of $\rho(\zeta)$ and fix a constant speed parametrization $\alpha: [0,\tau]\to \beta$ such that $\rho(\zeta)\alpha(0)=\alpha(\tau).$ We define mappings $f_r$ to be equivariant harmonic maps on $(\tilde{\Sigma}_r,\tilde{\mu})$ with equivariant boundary values specified by $\alpha$ on $\partial\tilde{\Sigma}_r$. In keeping with the notation of \cite{S}, we set $\varphi=f_2$, extend $\varphi$ vertically into the cusp by $\varphi(x,y)=\alpha(x)$ in a fundamental domain for $\Gamma$, and extend equivariantly. In this cusp, $\varphi$ satisfies
\begin{equation}\label{12}
    e(\mu,\varphi) = \frac{\ell^2}{2\tau^2}.
\end{equation}
 In \cite[Proposition 5.1]{S}, we show there is a uniform bound
 \begin{equation}\label{10}
     \int_{\Sigma_s} e(\mu,f_r) dA_\mu \leq \int_{\Sigma_s} e(\mu,\varphi) dA_\mu = \int_{\Sigma_2}e(\mu,\varphi) dA_\mu + \frac{(s-2)\ell^2}{2\tau}.
 \end{equation}
 Convergence on compacta in the $C^{\infty}$ topology then follows easily: the bound on the total energy gives bounds on the energy density, and via the elliptic theory one promotes to estimates on higher derivatives. One then use some now well-understood arguments to show that if $\rho$ is irreducible, all maps take a compact set into the same compact set, at which point one can use Arz{\`e}la-Ascoli and a diagonal argument. In the reducible case we may need to first renormalize $f_r$ via well-chosen translations $g_r$. Note that if we choose a different constant speed parametrization $\alpha$ for our approximation maps, we still get convergence to the same harmonic map. 
 
Next we turn to finiteness of $\E$. If no monodromy is hyperbolic, then harmonic maps have finite energy. So suppose $\rho(\zeta)$ is hyperbolic.
\begin{prop}\label{asym}
Let $f:(\tilde{\Sigma},\tilde{\mu})\to (X,\nu)$ be a $\rho$-equivariant harmonic map. Then there exists $C,c,y_0>0$ such that, in the cusp coordinates (\ref{17}), for all $y\geq y_0$, the inequality $$|e(\mu^f,f_\mu)(x,y)-\ell^2/2\tau^2|<Ce^{-cy}$$ holds.
\end{prop}
\begin{proof}
We implicitly work with the metric $\mu^f$. Write 
\begin{equation}\label{3}
    e= H +L = H^{1/2}L^{1/2}\Big ( \frac{H^{1/2}}{L^{1/2}} + \frac{L^{1/2}}{H^{1/2}}\Big )= |\Phi|(|\psi|^{-1}+|\psi|).
\end{equation}
 Since $\Phi$ has a pole of order at most $2$, changing coordinates to the cusp $[0,1]\times [1,\infty)$ (see Section 3.2 below) gives the expression 
 \begin{equation}\label{4}
     |\Phi| = \frac{\ell^2}{4\tau^2} + O(e^{-2\pi y}).
 \end{equation}
From \cite[page 513]{W}, if $f$ is Fuchsian, we have the estimate $$1-|\psi|^2 = O(e^{-\ell/2y}),$$ and hence $$1-O(e^{-c y}) \leq |\psi|^2 \leq 1.$$ If $\rho$ is not Fuchsian, we find the Fuchsian harmonic map $h:(\tilde{\Sigma},\tilde{\mu})\to (\mathbb{H},\sigma)$ with the same Hopf differential. Then from $H(f)L(f)=H(h)L(h)$ and \cite[Proposition 3.13]{S} we get $$L(h)\leq L(f), H(f)\leq H(h),$$ which implies $$|\psi(h)|\leq |\psi(f)|\leq |\psi(h)|^{-1},$$ and furthermore 
\begin{equation}\label{5}
    1-O(e^{-cy}) \leq |\psi|^2 \leq \frac{1}{1-O(e^{-cy})}.
\end{equation}
Inserting (\ref{4}) and (\ref{5}) into (\ref{3}) gives $$e(\mu^f,f) = \frac{\ell^2}{2\tau^2} + O(e^{-c_1y}),$$ as desired. 
\end{proof}
The following is now evident.
\begin{prop}\label{welldef}
For a fixed hyperbolic metric $\mu$, the integral defining $\E(\mu)$ is finite. That is, $\E : \mathcal{T}(\Gamma)\to\mathbb{R}$ is well defined.
\end{prop}
\end{subsection}
\begin{subsection}{Tangent vectors of Teichm{\"u}ller space.} To do analysis on Teichm{\"u}ller space, we need a tractable way to study tangent vectors. Fix a metric $\mu_0$ on $\Sigma$ and let $z=x+iy$ be a conformal coordinate for the compatible holomorphic structure, so that $\mu_0=\mu_0(z)|dz|^2$. In this subsection we describe variations of the metric $$\mu' = \mu + \dot{\mu}.$$ Since we are working with harmonic maps, which are conformally invariant, we are permitted to work in a specified conformal class. In particular, we may restrict to variations through complete finite volume hyperbolic metrics. The hyperbolic condition is satisfied if and only if $$\dot{\mu}_{11}+\dot{\mu}_{22}=0$$ and $$\phi(z)dz^2=(\dot{\mu}_{11}-i\dot{\mu}_{12})dz^2$$ is a holomorphic quadratic differential. When $\Sigma$ has a puncture, the necessary and sufficient condition on $\dot{\mu}$ to preserve the complete finite volume property is that $\phi$ has a pole of order at most 1 at the cusp.
\begin{remark}\label{Bers}
We have described the tangent space to Teichm{\"u}ller space at $(\Sigma,\mu)$ as the space of $L^1$-integrable quadratic differentials on the conjugate Riemann surface. This characterization coincides with the one we get from the Bers embedding into $\mathbb{C}^{3g-3+n}$. If $\mu(z)$ is a conformal metric, the mapping $$\phi(z)dz^2\mapsto \mu(z)^{-1}\phi(\overline{z})\frac{d\overline{z}}{dz}$$ yields the usual identification with the space of harmonic Beltrami forms.
\end{remark}
 We work out the growth condition on $\phi$ in the cusp coordinates for $U(\tau)$. For convenience put $\tau=1$. The mapping $z\mapsto w(z) = e^{2\pi iz}$ takes a vertical strip to a punctured disk $$\{0\leq x<1, y> h, y^{-2}|dz|^2\}\to \{0<|w|<e^{-h}, |w|^{-2}(\log |w|)^{2}|dw|^2\}$$ holomorphically and isometrically. A meromorphic quadratic differential in the disk with a pole of order at most $1$ is written $$\Phi = \phi(w) = (a_{-1}w^{-1}+\varphi(w))dw^2,$$ with $\varphi$ holomorphic. Applying the above holomorphic mapping, the differential transforms according to 
 \begin{equation}\label{ivreallylosttrack}
     \Phi =  \phi(w(z))\Big ( \frac{\partial w(z)}{\partial z}\Big )^2 dz^2= \phi(e^{iz})(ie^{iz})^2 dz^2= -(a_{-1} e^{iz} + e^{2iz}\varphi(e^{iz}) )dz^2.
 \end{equation}
Thus, any admissible variation decays exponentially in the cusp as we take $y\to \infty$.

We also need to describe the inverse variation $$(\mu')^{\alpha\beta} =\mu^{\alpha\beta}+\dot{\mu}^{\alpha\beta}.$$ In the Einstein notation, the relation $\mu^{\alpha\beta}\mu_{\beta\gamma} = \delta_{\alpha\gamma}$ gives $\dot{\mu}^{\alpha\beta}\mu_{\beta\gamma} + \mu^{\alpha\beta}\dot{\mu}_{\beta\gamma} = 0$, and hence $$\dot{\mu}^{\alpha\beta}=-\mu^{\alpha\rho}\mu^{\beta \tau}\dot{\mu}_{\tau\rho}.$$ When $\mu$ is conformal, this returns $$\dot{\mu}^{\alpha\beta} = -\mu^{-2}\dot{\mu}_{\alpha\beta}.$$ We derive that, in the cusp coordinates, 
\begin{equation}\label{11}
    \dot{\mu}^{\alpha\beta} = -y^4\dot{\mu}_{\alpha\beta},
\end{equation}
so the decay is still exponential. From these descriptions we deduce the following.
\begin{prop}\label{fin}
Suppose $g$ is a finite energy equivariant map with respect to a finite volume hyperbolic metric $\mu$. Then it also has finite energy for any other metric.
\end{prop}
\begin{proof}
By conformal invariance of energy, we are permitted to work with complete finite volume hyperbolic metrics. It is enough to prove the claim for metrics that are as close as we like to $\mu$. For then we can connect $\mu$ to any other metric $\mu'$ via a smooth path in the Teichm{\"u}ller space, cover this path with finitely many small balls, and argue inductively. That is, we can assume $\mu'=\mu+\dot{\mu}$, for some small variation $\dot{\mu}$. In a local coordinate $z=x+iy$, we write $|\mu+\dot{\mu}|=\textrm{det}(\mu+\dot{\mu})$, so that the volume form is $$dA_{\mu+\dot{\mu}}=\sqrt{|\mu+\dot{\mu}|}dz\wedge d\overline{z}=\sqrt{|\mu|-|\phi|^2}dz\wedge d\overline{z},$$ where $\phi$ is the holomorphic quadratic differential associated to $\dot{\mu}$. For simplicity, let's restrict to a $C^1$ map $g$. In a cusp, 
\begin{align*}
    2e(\mu+\dot{\mu},\nu, g)\sqrt{|\mu+\dot{\mu}|} &= \sqrt{|\mu|-|\phi|^2}(\mu+\dot{\mu}^{\alpha\beta})e_{\alpha\beta}(g) \\
 &= \frac{\sqrt{|\mu|-|\phi|^2}}{\sqrt{|\mu|}}\cdot \sqrt{|\mu|}e(\mu,g) + \sqrt{|\mu|-|\phi|^2} \dot{\mu}^{\alpha\beta} e_{\alpha\beta}(g).
\end{align*}
By hypothesis, the first term is uniformly bounded and converges to an integrable quantity. For the second term, using the flat cylinder metric we gather
$$\int_\Sigma (e_{11}+e_{22}) dA_{\mu^f} < \infty.$$ By Cauchy-Schwarz, $e_{12}$ is integrable as well. Meanwhile, the factor $\sqrt{|\mu|-|\phi|^2}\dot{\mu}^{\alpha\beta}$ decays exponentially as we go into the cusp.
\end{proof}
The proof above shows that the bound depends only on the energy of $g$ with respect to $\mu$ and the Teichm{\"u}ller distance of the new metric to $\mu$ (see \cite[Chapter 5]{Ah} for the definition).
\end{subsection}

\begin{subsection}{Variations: finite energy harmonic maps}
 In the proof of Proposition \ref{derivative}, we need to know that for a variation of hyperbolic metrics $t\mapsto \mu_t$, $f_{\mu_t}\to f_{\mu}$ pointwise as $t\to 0$. In this subsection and the next, we show uniform convergence of harmonic maps on compacta. To do so, we verify that analytic results for harmonic maps can be made uniform in the source metric. We start with the case of elliptic and parabolic monodromy at the cusp. Throughout, let $\rho:\Gamma\to G$ be a reductive representation.
\begin{prop}\label{parun}
 Assume the monodromy is elliptic or parabolic and fix an admissible metric $\mu_0$. For all $\mu$, there exists a $C_{k}>0$ depending only on the Teichm{\"u}ller distance from $\mu_0$ to $\mu$ such that for all $k>0$, $$|(\nabla^{\mu^f,\nu})^{(k)} df_\mu|_{\mu^f}\leq C_{k}.$$
\end{prop}
\begin{proof}
 We showed in \cite[Proposition 3.8]{S} that one can always find finite energy harmonic maps. Take such a map $g:(\tilde{\Sigma},\tilde{\mu}_0)\to (X,\nu)$, which by Lemma \ref{fin} is finite energy for any other metric, with a bound depending on the Teichm{\"u}ller distance to $\mu_0$. By the energy minimizing property in negatively curved spaces, $$\int_{\Sigma} e(\mu^f,f_{\mu}) dA_{\mu^f}=\int_{\Sigma} e(\mu,f_{\mu}) dA_\mu \leq \int_{\Sigma} e(\mu, g) dA_\mu \leq C.$$
For the uniform bounds on the energy density, independent of $\mu$, one uses a Harnack-type inequality, say, from \cite[page 171]{SY}, that only depends on uniform quantities: the Ricci curvature and the injectivity radius. Since we work with the flat-cylinder metric $\mu^f$ as opposed to the hyperbolic metric $\mu$, we do have uniform control on the injectivity radius. The estimates on higher order derivatives then come from the elliptic theory on the Sobolev space adapated to $(\Sigma,\mu_0)$ (see \cite[Chapter 10]{Ni}). It is clear from the general theory that the implicit constants in these estimates can be made uniform in $\mu$. 
\end{proof}
\begin{lem}\label{don}
Suppose $\rho$ is irreducible. Let $K\subset \tilde{\Sigma}$ be compact and suppose $f:(K,\mu)\to (X,\nu)$ is $C$-Lipschitz for some $C>0$. Then there exists a compact set $\Omega(K,C)\subset X$ that does not depend on the map $f$ such that $f(K)\subset \Omega(K,C)$.
\end{lem}
This is essentially carried out in the proof of Theorem 5.1 in \cite{S} (which is a modification of the argument from the main theorem of \cite{D}), but with some slightly different assumptions. We sketch a proof for the reader's convenience.
\begin{proof}
We show the claim holds for a single point $\xi$, and then the Lipschitz control promotes the result to general compact sets. Since $\rho$ is irreducible, there exists $\gamma\in \Gamma$ such that $\rho(\gamma)\xi\neq \xi$. For each $x\in K$, $\gamma$ may be represented by a loop $\gamma_x:[0,L_x]\to \Sigma$ based at $\pi(x)$. Since $K$ is compact, there is an $L>0$ such that $L_x\leq L$ for all $x$. Now choose a neighbourhood $B_\xi\subset X\cup\partial_\infty X$ such that $$d_\nu(X\cap B_\xi, \rho(\gamma) X\cap B_\xi) > CL.$$ This implies that for any $x\in K$, $f(x)$ cannot lie in $B_\xi$. Repeating this procedure and using compactness of $\partial_\infty X$, we get a neighbourhood of $\partial_\infty X$ in $X\cup \partial_\infty X$ that $f(K)$ cannot enter. We then take $\Omega(K,C)$ to be the complement of this neighbourhood. 
\end{proof}

\begin{lem}\label{compactvar1}
Let $K\subset \Sigma$ be compact. Then there is a choice of harmonic maps $f_{\mu}$ that vary continuously on lifts of $K$ inside $(\tilde{\Sigma},\tilde{\mu})$ in the $C^\infty$ topology.
\end{lem}
\begin{proof}
We argue by contradiction. First if $\rho$ is irreducible, suppose there exists $\delta>0$ and sequences $(k_n)_{n=1}^\infty\subset K$, $r_n\to \infty$, and $\mu_n\to\mu$ such that 
\begin{equation}\label{8}
    d(f_{\mu_n},f_{\mu})\geq \delta
\end{equation}
 for all $n$. Then by Lemma \ref{parun} we have uniform derivative bounds on each $f_{\mu_n}$ and by Lemma \ref{don} they all take a lift of $K$ to $\tilde{\Sigma}$ into a compact subset of $(X,\nu)$. By Arz{\`e}la-Ascoli we see the $f_{\mu_n}$ $C^\infty$-converge in $K$ along a subsequence to a limiting harmonic map $f_\infty$. By equivariance, we have this same convergence on the whole preimage of $K$. 

We now show $f_\infty=f$, which contradicts (\ref{8}). Taking a compact exhaustion of $\Sigma$ and applying the same argument on each compact set, the maps $f_{\mu_n}$ subconverge on compact subsets of $\tilde{\Sigma}$ to a limiting finite energy harmonic map $f_\infty'$ that agrees with $f_\infty$ on lifts of $K$. By uniqueness for finite energy maps we get $f_\infty=f$.

If $\rho$ is reducible, we can recenter the harmonic maps via translations along the geodesic so that they take $K$ into a fixed compact set (see the proof of \cite[Propistion 5.1]{S} for this routine procedure), and then repeat the argument above.
\end{proof}
\end{subsection}

\begin{subsection}{Variations: infinite energy harmonic maps}
Now we treat harmonic maps for representations $\rho$ with hyperbolic monodromy at the cusp. Recall the map $\varphi$ from Section 3.1, which for a metric $\mu$ we now denote $\varphi_\mu:(\tilde{\Sigma}_2,\tilde{\mu})\to (X,\nu)$. We emphasize that the boundary curve for $(\tilde{\Sigma_2},\tilde{\mu})$ is varying with $\tilde{\mu}$ 
\begin{lem}\label{approxcon}
Near a base hyperbolic metric $\mu_0$, $\varphi_\mu$ can be chosen so that the association $\mu\to \varphi_\mu$ is continuous in the $C^\infty$ topology .
\end{lem}
\begin{proof}
Since the metrics are varying smoothly, the lifts of $\mu$-horocycles vary smoothly with $\mu$. Indeed, in conformal cusp coordinates as in (\ref{9}), the $\mu$-horocycles are just curves with $y_\mu$ constant. Smoothness here thus comes from the regularity theory for the Beltrami equation (if we follow the approach of Ahlfors and Bers for finding isothermal coordinates on a hyperbolic surface). 

From the standard arguments (Section 3.1), the result amounts to choosing $\mu\to \varphi_\mu$ so that we have a uniform total energy bound on compacta, independent of $\mu$, and such that the boundary data varies continuously. It would follow that as $\mu\to \mu_0$, the harmonic maps subconverge to a harmonic map, and continuity in the boundary values shows this is exactly $\varphi$. One can then modify the contradiction argument from Lemma \ref{compactvar1} to see $C^\infty$ convergence on compacta along the whole sequence. 

By the energy minimizing property, it suffices to construct a family of maps $\psi_\mu$ with suitably chosen boundary values and uniformly controlled energy. To build $\psi_\mu$, let $f^\mu$ be the unique quasiconformal diffeomorphism between $(\Sigma_2,\mu)$ and $(\Sigma_2,\mu_0)$ that, in the cusp coordinates, takes $(i2,\tau_\mu+i2,\infty)\mapsto (i2,\tau_{\mu_0}+i2,\infty)$. Then $\psi_\mu=\varphi_{\mu_0} \circ f^\mu$ has the correct boundary values. By our choice of normalizations, $f^\mu$ converges to the identity as $\mu\to \mu_0$. This gives an energy bound on $f^\mu$, and moreover we get uniform bounds for $\psi_\mu$.
\end{proof}
We can now prove the analogue of Lemma \ref{compactvar1} for infinite energy harmonic maps.
\begin{lem}\label{vary}
Let $K\subset \Sigma$ be compact. Then there is a choice of harmonic maps $f_{\mu}$ that vary continuously on lifts of $K$ in the $C^\infty$ topology.
\end{lem}
\begin{proof}
From the estimate (\ref{10}) and the Fatou lemma we get 
$$\int_{\Sigma_s} e(\mu,f_\mu) dA_\mu \leq \int_{\Sigma_2} e(\mu,\varphi_\mu) dA_\mu + \frac{(r-2)\ell^2}{2\tau_\mu}.$$ Via the lemma above,
$$\int_{\Sigma_s} e(\mu,f_\mu) dA_\mu \leq C + \frac{2(r-2)\ell^2}{2\tau_{\mu_0}},$$ for $\mu$ close enough to $\mu_0$. Then from the discussion in Proposition \ref{parun}, we get uniform control on all derivatives on $K$. 

If $\rho$ is irreducible we apply Lemma \ref{don}, and if $\rho$ is reducible we rescale by translations along the geodesic so that they take $K$ into a fixed compact set. The argument from Section 3.1 then shows that for any sequence $\mu_n\to \mu$, the harmonic maps subconverge on compacta in the $C^\infty$ sense to a limiting harmonic map $f_\infty$. From \cite[Lemma 5.2]{S}, $d(f_{\mu_n},\varphi_{\mu_n})$ is uniformly bounded, and an investigation of the proof of this lemma shows it is maximized on $\Sigma_2$. By uniform convergence on compacta, $d(f_{\infty},\varphi)$ is also (non-strictly) maximized on $\Sigma_2$, and hence it is globally finite. According to the classification in Theorem \ref{S} and the explanation of the twist parameter, $f_\mu$ is the only harmonic map with this property: if we precompose an equivariant map with a non-trivial fractional Dehn twist, the distance between the original map and the new map grows without bound as we limit toward a lift of the puncture on the boundary at infinity.
\end{proof}
We deduce the following.
\begin{lem}\label{hypun}
Given a metric $\mu_0$, there exists a uniform $C_k>0$ such that $$|(\nabla^{\mu^f,\nu})^{(k)}df_\mu|_{\mu^f} \leq C_k$$ everywhere on $\Sigma$, with $C_k$ depending on the Teichm{\"u}ller distance to $\mu_0$.
\end{lem}
\begin{proof}
 By uniform convergence on compacta, we have uniform control on the distance to $\varphi_\mu$. We then couple the uniform energy bounds on $\varphi_\mu$ with Cheng's lemma to get uniform energy bounds, and then we appeal to the elliptic theory (as we have done many times at this point).
\end{proof}
\end{subsection}
\begin{subsection}{Energy minimizing properties}
It is well known that finite energy equivariant harmonic maps minimize the total energy among other equivariant maps. Here we show that this extends in some sense to the infinite energy setting. 
\begin{defn}\label{convalpha}
Let $g:(\tilde{\Sigma},\tilde{\mu})\to (X,\nu)$ be $\rho$-equivariant and let $\alpha$ be a parametrization of the geodesic axis of the image of the peripheral, say $\beta$. We say $g$ converges at $\infty$ to $\alpha$ if, after precomposing with the conformal mapping $i_r:\mathcal{C}\to \tilde{\Sigma}_r\backslash\tilde{\Sigma}_{r-1}$, $g\circ i_r:\mathcal{C}\to \beta$ converges in $C^0$ as $r\to\infty$ to a mapping $\tilde{\alpha}:\mathcal{C}\to\beta$ given by 
\begin{equation}\label{6}
    \tilde{\alpha}(x,y)=\alpha'(x),
\end{equation}
where $\alpha'(x)$ is some translation of $\alpha$ along $\beta$.
\end{defn}
\begin{prop}\label{min}
 Let $f$ be the harmonic map from Theorem \ref{S} whose Hopf differential has real residue at the cusp. Suppose that a locally Lipschitz map $g:(\tilde{\Sigma},\tilde{\mu})\to (X,\nu)$ converges at $\infty$ to $\alpha$. Then $e(\mu,f)-e(\mu,g)$ is integrable with respect to $\mu$ and $$\int_\Sigma e(\mu,f)-e(\mu,g) dA_\mu \leq 0.$$
\end{prop}
\begin{remark}
The lemma holds provided $g$ is weakly differentiable and these weak derivatives are locally $L^2$. If $g$ is not harmonic, then this inequality is strict.
\end{remark}
\begin{proof}
 Let $g_r$ be the harmonic map with the same boundary values as $g$ on $\Sigma_r$ and set $\phi:(\tilde{\Sigma},\tilde{\mu})\to \mathbb{R}$ to be any $\Gamma$-invariant function that is equal to $\mu^{-1}\ell^2/2\tau_\mu$ in $\mu$-conformal coordinates after $\tilde{\Sigma}_2$. For example, we can take the $\mu$-energy of the harmonic map $\varphi$. Then $$\int_{\Sigma_r}e(\mu, g_r) dA_\mu \leq \int_{\Sigma_r}e(\mu,g)dA_\mu,$$ and hence 
 \begin{equation}\label{7}
     \liminf_{r\to\infty}\int_{\Sigma_r}e(g_r)-\phi dA_\mu \leq \int_{\Sigma}e(g)-\phi dA_\mu.
 \end{equation}
Without loss of generality, $\alpha$ is the limiting parametrization of the geodesic for $g$. Let $\varphi_r$ be the unique harmonic map on $\tilde{\Sigma}_r$ with boundary values $\alpha$ on $\partial \tilde{\Sigma}_r$, defined by using $\alpha$ on one lift and then extending equivariantly. The distance function $$p\mapsto d(\varphi_r(p),g_r(p))$$ is subharmonic and hence maximized on $\partial \Sigma_r$. By the convergence property of $g$, we can thus assume that for $r$ large enough, $$d(\varphi_r(p),g_r(p))\leq 1$$ for all $p$. By Cheng's lemma, we obtain uniform bounds on the energy density of $g_r$ in terms of that of $\varphi_r$ on compact sets. Since $\varphi_r\to f$ locally uniformly on compacta, $g_r$ also subconverges locally uniformly on compacta, and the limiting map is harmonic. As $g_r$ has bounded distance to $\varphi$, as above the uniqueness result Theorem \ref{S} shows the limit must be $f$. Returning to our integrals (\ref{7}), Fatou's lemma then yields 
\begin{align*}
    \int_\Sigma e(\mu,f) - \phi dA_\mu &= \int_{\Sigma_2} e(\mu,f) - \phi dA_\mu + \int_{\Sigma\backslash\Sigma_2} e(\mu,f)-\phi dA_\mu\\
    & = \int_{\Sigma_2} e(\mu,f) - \phi dA_\mu + \int_{2}^\infty \Big(\int_0^{\tau_\mu} e(\mu^f,f)(x,y) - \phi(x,y) dx\Big)dy \\
    &\leq \liminf_{r\to\infty }  \int_{\Sigma_r}e(g_r)-\phi dA_\mu. 
\end{align*}
(\ref{7}) then gives $$\int_\Sigma e(\mu,f) - \phi dA_\mu \leq \int_{\Sigma}e(g)-\phi dA_\mu.$$
If the integral on the right is infinite then the result of our lemma is obvious, and if it is finite then we can rearrange to get the desired inequality.
\end{proof}
\begin{lem}\label{close}
In the setting above, work with a metric $\mu_0$. Then for any other $\mu$, $f_\mu$ satisfies the hypothesis above. 
\end{lem}
\begin{proof}
Let $w$ be a complex coordinate parametrizing a cusp as a quotient of a vertical strip of length $1$. Recall $z$ is defined on a strip $y\geq a$, $0\leq x \leq \tau_\mu$, and similar for $z_\mu$ (we take the same $a$). We can assume $z=f^{\lambda_0}(w)$, $z_{\mu}=f^{\lambda}(w)$, where $f^{\lambda_0}$, $f^\lambda$ are smooth quasiconformal maps with complex dilatations $\lambda_0,\lambda$ respectively. From the choice of coordinates, $f^{\lambda_0}$ maps $(0,1,\infty)\mapsto (0,\tau_{\mu_0},\infty)$ and $f^{\lambda}$ maps $(0,1,\infty)\mapsto (0,\tau_\mu,\infty)$. $f^\lambda \circ (f^{\lambda_0})^{-1}$ takes $(x,y)\mapsto (x_\mu,y_\mu)$, yielding a quasiconformal mapping between the strips. Provided the metrics are close enough in Teichm{\"u}ller space, $\lambda$ and $\lambda_0$ are related by a small variation $$\lambda = \lambda_0 + \dot{\lambda}.$$ Note that from the computation in Section 3.2, $\dot{\lambda}\to 0$ as $y\to \infty$. It follows that $f^{\lambda}\circ (f^{\lambda_0})^{-1}$ has complex dilatation tending to $0$ as $y\to\infty$. Mori's theorem  \cite[Chapter 3]{Ah} implies $f^{\lambda}\circ (f^{\lambda_0})^{-1}$ has uniform H{\"o}lder continuity.

Upon pre-composing with cylinders that are conformal for $\mu$, the harmonic map $f_\mu$ converges to a projection onto the geodesic axis $\beta$. We can take these $\mu$-cylinders as large as we like. Using the H{\"o}lder continuity and our normalizations, given a family $\mu_0$ cylinders of height $1$, we can embed each cylinder in a $\mu$-cylinder of fixed height. $C^0$ convergence to the projection onto the geodesic thus follows.
\end{proof}
\end{subsection}

\begin{subsection}{Proof of Proposition \ref{derivative}}
For the remainder of this section, fix a background metric $\mu_0$. There is a Serre duality pairing between holomorphic quadratic differentials on $\Sigma$ with at most first order poles at the cusp and harmonic Beltrami forms with appropriate decay (see Remark \ref{Bers}): $$\langle \Phi,\psi\rangle=\int_\Sigma \phi \psi dz d\overline{z},$$ where $\Phi=\phi(z)dz^2$ is a coordinate expression. We recall that Proposition \ref{derivativefin} asserts that the derivative of the total energy of a finite energy harmonic map, evaluated on a harmonic Beltrami form $\psi$, is $$dE_\rho[\mu](\psi) = -4\Real \langle \Phi(f_\mu),\psi\rangle.$$ For finite energy harmonic maps, Proposition \ref{derivative} is a direct corollary of Proposition \ref{derivativefin}, which we prove first. For closed surfaces, Proposition \ref{derivativefin} is well known, although the history is unclear---the earliest computation may go back to the work of Douglas \cite{Do}. The most modern version is contained in \cite{We}. The difference without compactness is that we must control variations at the cusp.
\begin{proof}[Proof of Proposition \ref{derivativefin}]
Let us assume $\mu(z)$ is conformal, and let $$\mu_t = \mu + t\dot{\mu} + t\epsilon(t)$$ be a variation through hyperbolic metrics, where $\epsilon(t)\to 0$ as $t\to 0$. We denote by $\phi_t$ the associated holomorphic quadratic differential, which is given in local coordinates by $$\phi_t(z) = t\Big( (\dot{\mu}_{11}+\dot{\epsilon}_{11}(t)) - i(\dot{\mu}_{12}+\dot{\epsilon}_{12}(t))\Big ) dz^2,$$ and we set $\psi_t$ to be the associated harmonic Beltrami form. We put $\phi$ to be the quadratic differential for the variation $\mu+\dot{\mu}$, and $\psi$ the harmonic Beltrami form. Writing $f=f_\mu$, $f_t=f_{\mu_t}$, our objective is to show that $$\frac{d}{dt}|_{t=0} \int_{\Sigma} e(\mu_t,f_t) dA_{\mu_t} = -4\Real\langle \Phi(f),\psi\rangle.$$ By energy minimization, we have the inequalities $$\int_\Sigma e(\mu_t,f_t) dA_{\mu_t} - \int_\Sigma e(\mu,f) dA_{\mu} \leq \int_\Sigma e(\mu_t,f) dA_{\mu_t} - \int_\Sigma e(\mu,f) dA_{\mu}$$ and $$\int_\Sigma e(\mu_t,f_t) dA_{\mu_t} - \int_\Sigma e(\mu,f) dA_{\mu} \geq \int_\Sigma e(\mu_t,f_t) dA_{\mu_t} - \int_\Sigma e(\mu,f_t) dA_{\mu}.$$
Thus, it suffices to divide by $t$ and take the limit on the two expressions on the right. We expand
\begin{equation}\label{idknow}
    \frac{\sqrt{|\mu|_t}\mu_t^{\alpha\beta} - \sqrt{|\mu|}\mu^{\alpha\beta}}{t}= \Big( \frac{\sqrt{|\mu|-|\phi_t|^2}-\sqrt{|\mu|}}{t}\Big ) \mu^{\alpha\beta}+ \sqrt{|\mu|-|\phi_t|^2}(\dot{\mu}^{\alpha\beta}+\epsilon^{\alpha\beta}(t)).
\end{equation}
 Twice the first integrand is obtained by hitting the expression above with $e_{\alpha\beta}(f)$, and the second with $e_{\alpha\beta}(f_t)$, which both converge to $e_{\alpha\beta}(f)$ pointwise as $t\to 0$. The first term in (\ref{idknow}) converges to $0$, while the second one to $\sqrt{|\mu|}\dot{\mu}^{\alpha\beta}$. Using the relation $\dot{\mu}^{\alpha\beta}=-\mu^{-2}\dot{\mu}_{\alpha\beta}$, the integrand converges pointwise to 
$$-\frac{1}{2\mu} (\dot{\mu}_{11} e_{11} + \dot{\mu}_{22}e_{22} + 2\dot{\mu}_{12}e_{12})= -\frac{1}{2\mu}(\dot{\mu}_{11}(e_{11}-e_{22}) + 2\dot{\mu}_{12}e_{12}) = -4\Real \Phi(f)\psi.$$
Therefore, it suffices to show the integrands are always bounded above by an integrable quantity, for then we can justify an application of the dominated convergence theorem. We have the expression $$\frac{\sqrt{|\mu|-|\phi_t|^2}-\sqrt{|\mu|}}{t} = \frac{-|\phi_t|^2}{t(\sqrt{|\mu|-|\phi_t|^2}+\sqrt{|\mu|})} = \frac{-t\Big| (\dot{\mu}_{11}+\dot{\epsilon}_{11}(t)) - i(\dot{\mu}_{12}+\dot{\epsilon}_{12}(t))\Big |^2}{\sqrt{|\mu|-|\phi_t|^2}+\sqrt{|\mu|}}.$$ If $a_t$ is the $-1$ Laurent coefficient in the coordinates on the punctured disk (Section 3.2) for the quadratic differential associated to $\dot{\mu}+\epsilon(t)$, then because $\epsilon(t)\to 0$, $a_t$ converges to a constant $a_0$. Therefore, for $t$ small enough, using the expression (\ref{ivreallylosttrack}), we see that for small $t$, the first term in (\ref{idknow}) decays at most like $$-t(2|a_0|+1)^2y^2e^{-2y}y^2\sim y^4e^{-2y}.$$ By similar reasoning, the second term in (\ref{idknow}) decays like $y^2e^{-2y}$. Thus, it suffices to bound $e_{\alpha\beta}(f_t)$ by a constant. The one obstruction to applying Lemma \ref{parun} for $f_t$ is that $e_{\alpha\beta}$ depends on the cusp coordinates for $\mu$ rather than $\mu_t$. The argument of Lemma \ref{close} shows that the $\mu$ and $\mu_0$ coordinates are related by a mapping that is asymptotically Lipschitz---the H{\"o}lder exponent in Mori's theorem is the reciprocal of the quasiconformal dilatation, which is tending to $1$. Thus, once high enough in the cusp, we have the same bound in the $\mu_0$-coordinates. From the discussion above, the result follows.
\end{proof}
\begin{remark}\label{npc}
We have worked in negative curvature, but the proof carries through in non-positive curvature if total energies of harmonic maps to $(X,\nu)$ are unique, and if one can locally continuously associate source metrics to harmonic maps in the $C^\infty$ topology. Our proof uses the existence of a single finite energy harmonic map and an energy minimizing property. In non-positive curvature we have energy minimization, and for existence one can modify our \cite[Proposition 3.8]{S} to include NPC manifolds.
\end{remark}

Now we turn to the main result of this section. Assume the monodromy is hyperbolic. 
\begin{proof}[Proof of Proposition \ref{derivative}]
As above, let $\mu_t=\mu+ t\dot{\mu}+ t\epsilon(t)$ be the variation with $\mu$. Similarly, denote by $\phi_t$ and $\psi$ the family of holomorphic quadratic differentials and the harmonic Beltrami form associated to this path respectively. Put $f=f_\mu$, $f_t=f_{\mu_t}$, $h=h_\mu$, $h_t=h_{\mu_t}$, recalling that $f$ and $h$ correspond to the representations $\rho_1:\Gamma\to\PSL(2,\mathbb{R})$, $\rho_2:\Gamma\to G$ respectively. From the energy minimization Lemma \ref{min}, we deduce
$$\int_\Sigma e(\mu,h)-e(\mu, f_t) dA_\mu \leq \E(\mu) \leq \int_\Sigma e(\mu,h_t)-e(\mu, f) dA_\mu$$ and $$\int_\Sigma e(\mu_t,h_t)-e(\mu_t, f) dA_{\mu_t} \leq \E(\mu_t) \leq \int_\Sigma e(\mu_t,h)-e(\mu_t, f_t) dA_{\mu_t}.$$ We cannot apply Lemma \ref{asym} to an expression like $e(\mu,f_t)$ directly, but from the expression for the cusp coordinates (\ref{17}), it follows that when changing $\sqrt{|\mu|}e(\mu,f)$ to $\sqrt{|\mu_t|}e(\mu_t,f_\mu)$, the energy density is asymptotically multiplied by $\tau_{\mu_t}/\tau_{\mu}$. Therefore, Lemma \ref{asym} does imply that every integral above is finite. Furthermore, it makes sense to manipulate these integrals, and so the difference $\E(\mu_t)-\E(\mu)$ is bounded above by $$\int_\Sigma (\sqrt{|\mu_t|}e(\mu_t,h)-\sqrt{|\mu|}e(\mu,h) ) - (\sqrt{|\mu_t|}e(\mu_t,f_t)-\sqrt{|\mu|}e(\mu,f_t))dzd\overline{z}$$ and bounded below by $$\int_\Sigma (\sqrt{|\mu_t|}e(\mu_t,h_t) - \sqrt{|\mu|}e(\mu,h_t)) - (\sqrt{|\mu_t|}e(\mu_t,f) - \sqrt{|\mu|}e(\mu,f))dzd\overline{z}.$$
Dividing by $t$, the local computations from the proof of Proposition \ref{derivativefin} work out almost the same. We write out the details for the upper bound and leave the lower bound to the reader. The relevant quotient may be expressed
$$\frac{1}{2}\Big ( \Big (\frac{\sqrt{|\mu|-|\phi_t|^2}-\sqrt{|\mu|}}{t}\Big )\mu^{\alpha\beta} + \sqrt{|\mu|-|\phi_t|^2}(\dot{\mu}^{\alpha\beta}+\epsilon^{\alpha\beta}(t))\Big )(e_{\alpha\beta}(h)-e_{\alpha\beta}(f_t)).$$
The quotient converges pointwise to $-4\Real(\Phi(h)-\Phi(f))\cdot \overline{\psi}$, so we are left to bound it by an integrable quantity that does not depend on $t$. Lemma \ref{hypun} and the fact that the change of coordinates is asymptotically conformal shows that the second term in the product is uniformly integrable. And it was shown in the proof of the previous proposition that the first term in the product decays exponentially.
\end{proof}
\end{subsection}
\begin{subsection}{Different twist parameters}\label{3.9}
The proof of Proposition \ref{derivative} is complete for finite energy harmonic maps, and infinite energy harmonic maps with $\theta=0$. We sketch the necessary adjustments for the remaining harmonic maps $f_\mu^\theta$. Fix $\theta\in\mathbb{R}$. To construct the harmonic map $f_\mu^\theta$, we define the fractional Dehn twist to be the map in the cusp coordinates defined by $$x+iy\mapsto x+\theta_iy+iy.$$ We then postcompose with the approximation maps $f_r$ and take limits on these maps as in Section 3.1 (see \cite[Section 5.1]{S}).  
\begin{itemize}
    \item The inequality in Lemma \ref{asym} becomes $$|e(\mu,f_\mu^\theta)-\Lambda(\theta)\ell(\rho(\gamma))^2/2\tau^2|<Ce^{-cy}.$$ The proof is the same, except the Hopf differential has expression at the cusp according to Theorem \ref{S}.
    \item For continuity of harmonic maps on compacta, the main step is to find bounds on the total energy of $\varphi_\mu^\theta$ on compacta. This is a consequence of the chain rule, since in the coordinates (\ref{17}), the derivative matrix of the fractional Dehn twist is simply 
    \begin{equation}\label{fdehn}
        \begin{pmatrix} 1 & \theta \\ 0 & 1 \end{pmatrix}.
    \end{equation}
    \item For the energy minimizing property, we can say that a map converges at $\infty$ to $\alpha^\theta$ if after pulling back to conformal cylinders approaching the cusps, the harmonic maps converge to the composition of a projection onto the geodesic with a fractional Dehn twist on the cylinder. We then have the analogue of Lemma \ref{min}, which we can show the harmonic maps $f^\theta$ satisfy.
\end{itemize}
 We leave the rest of the details to the reader.
\end{subsection}
\end{section}
\begin{section}{Existence and uniqueness of maximal surfaces}
With Proposition \ref{derivative} in hand, we prove the main theorems. For convenience we assume the twist parameter is zero---the proof has no dependence on it. The existence result is immediate from the next proposition.
\begin{prop}\label{proper}
The functional $\E=\E_{\rho_1,\rho_2}:\mathcal{T}(\Gamma)\to \mathbb{R}$ is proper if and only if $\rho_1$ almost strictly dominates $\rho_2$.
\end{prop}
Indeed, by Proposition \ref{derivative}, properness implies the existence of a maximal spacelike immersion. Thus, the proof of the existence result reduces to showing that almost strict domination implies properness.

\begin{subsection}{Mapping class groups}
Preparing for the proof of Theorem A, we review some Teichm{\"u}ller theory. Set $\textrm{Diff}^+(\Sigma)$ to be the group of $C^\infty$ orientation preserving diffeomorphisms of $\Sigma$, equipped with the $C^\infty$ topology. Denote by $\textrm{Diff}_0^+(\Sigma)$ the normal subgroup consisting of maps that are isotopic to the identity. The mapping class group of $\Sigma$ is defined $$\textrm{MCG}(\Sigma):=\pi_0(\textrm{Diff}^+(S))=\textrm{Diff}^+(S)/\textrm{Diff}_0^+(S).$$
Given a collection of boundary lengths $\ell=(\ell_1,\dots,\ell_n)\in\mathbb{R}^n$, we can consider the Teichm{\"u}ller space $\mathcal{T}_\ell(\Gamma)$ of surfaces with punctures and geodesic boundary with lengths determined by $(\ell_1,\dots,\ell_n)$. This is not a relative representation space but a union of them. The mapping class group acts on Teichm{\"u}ller space by pulling back classes of hyperbolic metrics: 
\begin{equation}\label{20}
    [\varphi]\cdot [\mu]\mapsto [\varphi^*\mu].
\end{equation}
The action of the mapping class group is properly discontinuous and the quotient is the moduli space of hyperbolic surfaces.

There is a wealth of metrics on Teichm{\"u}ller space, and the mapping class group acts by isometries on a number of them. In the work below, we set $d_\mathcal{T}$ to be any metric distance function on Teichm{\"u}ller space on which the mapping class group acts by isometries. One example is the Teichm{\"u}ller distance.

Selecting a basepoint $z\in \Sigma$, there is an action of the mapping class group on $\pi_1(\Sigma,z)$, which identifies with $\Gamma$. Given $\varphi\in \textrm{Diff}^+(\Sigma)$ such that $\varphi(z)=z$, the induced map $\varphi_*:\pi_1(\Sigma,z)\to\pi_1(\Sigma,z)$ is an automorphism. While a general $\varphi$ may not fix $z$, one can find a different $\varphi_1\in \textrm{Diff}^+(\Sigma)$ that does fix $z$ and is isotopic to $\varphi$. The isotopy gives rise to a path $\gamma_1$ from $\varphi(z)$ to $z$. If $\varphi_2\in \textrm{Diff}^+(\Sigma)$ also fixes $z$ and is isotopic to $\varphi$, then we get another path $\gamma_2$ from $\varphi(z)$ to $z$. The automorphisms $(\varphi_1)_*$ and $(\varphi_2)_*$ of $\pi_1(\Sigma,z)$ differ by the inner automorphism corresponding to the conjugation by the class of $\gamma_1\cdot \overline{\gamma}_2$. This association furnishes an injective homomorphism from $$\textrm{MCG}(\Sigma)\to \textrm{Out}(\pi_1(\Sigma))= \textrm{Aut}(\pi_1(\Sigma))/\textrm{Inn}(\pi_1(\Sigma)).$$ 
Through this injective mapping, the mapping class group acts on representation spaces: we can precompose a representative of a representation with a representative of an element in $\textrm{Out}(\pi_1(\Sigma))$, and then take the corresponding equivalence class. On a Teichm{\"u}ller space, this agrees with the action (\ref{20}). For this action, we use the similar notation $$[\varphi]\cdot [\rho]\mapsto [\varphi^*\rho]=[\rho\circ \varphi_*].$$ Note that it does not in general preserve relative representation spaces: a mapping class may permute the punctures. Finally, we record that $\textrm{MCG}(\Sigma)$ acts equivariantly with respect to 
\begin{itemize}
    \item length spectrum: for a representation $\rho$ and $[\varphi]\in \textrm{MCG}(\Sigma)$, $\ell(\varphi^*\rho(\gamma))=\ell(\rho(\varphi_*(\gamma)))$. If $\mu$ is a hyperbolic metric we set $\ell_\mu(\gamma)$ to be the $\mu$-length of the geodesic representative of $\gamma$ in $(\Sigma,\mu)$. A restatement of the above is $\ell_{\varphi^*\mu}(\gamma)=\ell_\mu(\varphi(\gamma))$.
    \item Energy of harmonic maps: if $f: (\tilde{\Sigma},\tilde{\mu})\to (X,\nu)$ is a $\rho$-equivariant map, then 
    \begin{equation}\label{25}
      e(\mu,f_\mu)=e(\varphi^*\mu,f_\mu\circ \varphi).  
    \end{equation}
 Furthermore, if $f$ is harmonic, then $f\circ \varphi: (\tilde{\Sigma},\tilde{\varphi}^*\tilde{\mu})\to (X,\nu)$ is a $\varphi^*\rho$-equivariant harmonic map.  
\end{itemize}
\end{subsection}

\begin{subsection}{Existence of maximal surfaces} Suppose a simple closed curve $\xi$ is either a geodesic boundary component or a horocycle in a hyperbolic surface. If $\xi$ is a boundary, let $C_d(\xi)$ denote the collar around $\xi$ consisting of points with distance to $\xi$ at most $d$. If it is a horocycle, put $C_d(\xi)$ to be the union of the $d$-collar and the enclosed cusp (both are defined for suitably small $d$).

We now begin the proof of properness. Suppose $\rho_1$ almost strictly dominates $\rho_2$. We assume there is a single peripheral curve---there are no substantial changes in the case of many cusps---and we set $\xi$ to be either the geodesic boundary component of the convex core of $\mathbb{H}/\rho_1(\Gamma)$, or a deep horocycle of $\mathbb{H}/\rho_1(\Gamma)$, depending on whether the monodromy is hyperbolic or parabolic. Let $g$ be an optimal $(\rho_1,\rho_2)$-equivariant map. If the image of the peripheral curve under $\rho_1$ and $\rho_2$ is hyperbolic, then $g$ has constant speed on the boundary components. Note that $g\circ h_\mu$ is $\rho_2$-equivariant, and in the hyperbolic case, converges at $\infty$ to a projection onto the geodesic at infinity in the sense of Definition \ref{convalpha}. Thus, from Lemma \ref{min} we have $$\E(\mu) = \int_\Sigma e(\mu,h_\mu) - e(\mu, f_\mu) dA_\mu\geq \int_\Sigma e(\mu,h_\mu) - e(\mu, g\circ h_\mu) dA_\mu.$$ One of the main advantages of replacing $f_\mu$ with $g\circ h_\mu$ is that the integrand is now non-negative. Moving toward the proof of properness, let $([\mu_n])_{n=1}^\infty \subset \mathcal{T}(\Gamma)$ be such that $$\E(\mu_n) \leq K$$ for all $n$, for some large $K$. It suffices to show $[\mu_n]$ converges in $\mathcal{T}(\Gamma)$ along a subsequence. Structurally, our proof is similar to that of the classical existence result for minimal surfaces in hyperbolic manifolds (see \cite{SY2} and also the more modern paper \cite{GW}): we first prove 1) compactness in moduli space, and then 2) extend to compactness in Teichm{\"u}ller space.

Toward 1), we show there is a $\delta> 0$ such that for all non-peripheral simple closed curves in $\Gamma$, we have $\ell_{\mu_n}(\gamma)\geq \delta$. We argue by contradiction: suppose this is not the case, so that there is a sequence of non-peripheral simple closed curves $(\gamma_n)_{n=1}^\infty\subset \Gamma$ such that $$\ell_n:=\ell_{\mu_n}(\gamma_n)\to 0$$ as $n\to\infty$. By the regular collar lemma, we know that in $(\Sigma,\mu_n)$, the geodesic representing $\gamma_n$ is enclosed by a collar $C_n$ of width $w_n = \textrm{arcsinh}((\sinh \ell_n)^{-1})$. Up to a conjugation of the holonomy, $C_n$ is conformally equivalent to $\{x+iy\in\mathbb{C}: 0\leq x \leq w_n, 0\leq y\leq 1\}$ with the lines $\{y=0\}$ and $\{y=1\}$ identified.

Let $f_n=f_{\mu_n}$, $h_n=h_{\mu_n}$. Henceforward, conformally modify the metric $\mu_n$ to be flat in $C_n$. Our control over the energy functional depends on our knowledge of the local Lipschitz constant of the map $g$. This in turns relies on the image of the harmonic map $h_n$, in particular how far it takes $C_n$ into $C_d(\xi)$. For all $x,0<t<d$ in question, set $A_x^t=\{\theta\in \{x\}\times S^1: h_n(x,y)\not \in C_t(\xi)\}$. We also write $$\textrm{Lip}(g,t)=\max_{x\in C(\mathbb{H}/\rho_1(\Gamma))\backslash C_t(\xi)}\textrm{Lip}_x(g).$$ We have the inequalities
\begin{align*}
    \E(\mu_n)(\sigma_n) &\geq \int_{C_n} e(\mu_n,h_n) - e(\mu_n,g\circ h_n) dA_{\mu_n} \\
    &= \int_0^{w_n}\int_0^1 e(\mu_n, h_n) - e(\mu_n, g\circ h_n) dy ds \\
    &\geq w_n\min_{x} \int_{A_x^t}  e(\mu_n, h_n) - e(\mu_n, g\circ h_n) dy\\
    &\geq w_n\min_x(1-\textrm{Lip}(g,t)^2)\int_{A_x^t}  e(\mu_n,h_n) dy \\
    &\geq \frac{w_n}{2}\min_x(1-\textrm{Lip}(g,t)^2)\int_{A_x^t}  \Big |\frac{\partial h_n(x,y)}{\partial y}\Big |^2 dy \\ 
    &\geq \frac{w_n}{2}\min_x(1-\textrm{Lip}(g,t)^2)(\ell(\textrm{A}_x^t))^{-1}\Big (\int_{A_x^t}  \Big |\frac{\partial h_n(x,y)}{\partial y}\Big |dy \Big )^2  \\
    &\geq \frac{w_n}{2}\min_x(1-\textrm{Lip}(g,t)^2)\Big (\int_{A_x^t}  \Big |\frac{\partial h_n(x,y)}{\partial y}\Big | dy \Big )^2
\end{align*}    
Here $\ell(A_x^t)$ is the length of the (possibly broken) segment $A_{x,t}$, which is clearly bounded above by $1$. Fuchsian representations have discrete length spectrum, and hence there is a $\kappa>0$ such that $\ell(\rho_1(\gamma))\geq \kappa$ for all $\gamma\in \Gamma$. 
\begin{lem}\label{nocollar}
Set $t=d/2$. Then for any $x$, the inequality $$\int_{A_x^{d/2}}|\frac{\partial h_n(x,y)}{\partial y}\Big | dy \geq \min\{d/2,\kappa\}=: k$$ holds.
\end{lem}
\begin{proof}
Let $\alpha$ be any core curve of the form $\{x\}\times S^1$. If $h_n(\alpha)$ always remains outside $C_{d/2}(\xi)$, then $$\int_{A_x^t} |\frac{\partial h_n(x,y)}{\partial y}\Big | dy= \int_{0}^1 \Big |\frac{\partial h_n(x,y)}{\partial y}\Big | dy \geq \ell(\rho_1(\gamma_n))\geq \kappa.$$ Thus, we assume $h_n(\alpha)$ intersects $C_{d/2}(\xi)$. We parametrize $h_n(x,\cdot)$ by arc length, so that the integral in question measures the hyperbolic length on $\mathbb{H}/\rho_1(\Gamma)$ of the segment of $h_n(\alpha)$ that does not enter $C_{d/2}(\xi)$. If this length is ever less than $d/2$, then the curve $h_n(\alpha)$ is contained in $C_d(\xi)$. Phrased differently, it is a simple closed curve contained in an embedded cylinder. There is only one such homotopy class of curves, namely the homotopy class of the boundary geodesic. This situation is impossible, since $h_n(\alpha)$ is homotopic to a non-peripheral simple closed curve, and hence the length is at least $d/2$.
\end{proof}
Returning to the inequalities above, we now have $$\E(\mu_n)\geq \frac{w_n}{2}(1-\textrm{Lip}(g,d/2)^2)k^2.$$ We thus find $\E(\mu_n)\to\infty$ as $n\to \infty$, which violates the uniform upper bound.

The lower bound on the length spectrum shows the metrics $(\mu_n)_{n=1}^\infty$ satisfy Mumford's compactness criteria \cite{Mu} and hence project under the action of the mapping class group to a compact subset of the moduli space. 

Now we promote to compactness in Teichm{\"u}ller space. By compactness in moduli space, after passing to a subsequence of the $\mu_n$'s, there exists a sequence of mapping class group representatives $(\psi_n)_{n=1}^\infty$ such that $\psi_n^*\mu_n$ converges as $n\to\infty$ to a hyperbolic metric $\mu_\infty$. 

\begin{lem}\label{evenconstant}
Upon passing to a further subsequence, $([\psi_n])_{n=1}^\infty\subset \textrm{MCG}(\Sigma)$ is eventually constant.
\end{lem}
The main thrust of the proof is to show that for each non-peripheral $\gamma\in \Gamma$, $\ell(\rho_1\circ \psi_n(\gamma))$ is uniformly bounded above in $n$. Our proof of this claim is by contradiction and similar in nature to the argument above. Suppose there is a class of curves $\gamma\in\Gamma$ such that $\ell(\rho_1\circ \psi_n(\gamma))\to \infty$ as $n\to\infty$. Modify the notation: set $h_n = h_{\mu_n}\circ \psi_n$ and $f_{\mu_n}\circ \psi_n$. These are harmonic for $\psi_n^*\rho_1$ and $\psi_n^*\rho_2$ respectively, and have the correct boundary behaviour according to Theorem \ref{S}. From (\ref{25}), we also have the equality $$\E_{\rho_1,\rho_2}(\mu_n) = \E_{\psi_n^*\rho_1,\psi_n^*\rho_2}(\psi_n^*\mu_n).$$ On each $(\Sigma,\psi_n^*\mu_n)$, there is a collar $C_n$ of finite width $w_n$ around $\gamma$, and the $C_n$'s converge to some collar $C_\infty$ in $(\Sigma,\mu_\infty)$ of width $w_\infty<\infty$. As before, we perturb to a flat metric and parametrize $C_n$ by $[0,w_n]\times S^1$.
We note that $g$ is $(\psi_n^*\rho_1,\psi_n^*\rho_2)$-equivariant. Therefore, we can apply our previous reasoning to see that for $n$ large enough, any small enough $x$, and $t\in [0,w_n]$, 
\begin{equation}\label{26}
    \E_{\rho_1,\rho_2}(\mu_n) = \E_{\psi_n^*\rho_1,\psi_n^*\rho_2}(\psi_n^*\mu_n) \geq  \frac{w_\infty}{4}\min_x(1-\textrm{Lip}(g,t)^2)\Big (\int_{A_x^t}  \Big |\frac{\partial h_n(x,y)}{\partial y}\Big | dy \Big )^2,
\end{equation}
 where $A_x^t$ is defined as above. This leads us to the analogue of Lemma \ref{nocollar}.
\begin{lem}\label{bigcurve}
Set $t=d/2$. Then, independent of the choice of $x$, $$\int_{A_x^{d/2}} \Big | \frac{\partial h_n(x,y)}{\partial y}\Big | dy \to \infty$$ as $n\to \infty$.
\end{lem}
\begin{proof}
Choose any simple closed curve $\alpha$ of the form $\{x\}\times S^1$. Parametrize $h_n(x,\cdot)$ by arc length so that the integral above returns the length of the segment of $h_n(\alpha)$ that does not enter $C_{d/2}(\xi)$. If $h_n(\alpha)$ does not enter $C_{d/2}(\xi)$, then we can argue as we did in the proof of Lemma \ref{nocollar}, bounding the length below by the $\ell(\psi_n^*\rho_1(\gamma))$, which blows up. Thus, we restrict our discussion to $\alpha$ such that $h_n(\alpha)$ intersects $C_{d/2}(\xi)$. Let us first assume there is a positive integer $K$ such that every $h_n(\alpha)$ enters $C_{d/2}(\xi)$ at most $K$ times. Of course, $h_n(\alpha)$ cannot live entirely in $C_{d/2}(\xi)$, for then it lies in the wrong homotopy class. We construct a new curve as follows: if we choose a basepoint not within $C_{d/2}(\xi)$, then every time $h_n(\alpha)$ enters $C_{d/2}(\xi)$, there is a corresponding point at which it exits.
We plan to erase the segment of $h_n(\alpha)$ that connects these two points and replace it with the one of the two possible paths on $\partial C_{d/2}(\xi)$, say, $\beta_1$ and $\beta_2$. If $h_n(\alpha)$ spends some time going in a path along the boundary circle when it is entering or when it is exiting (we count just touching it once as an exit), we connect the endpoint of the entrance path to the starting point of the exiting path and concatenate with the path going along the boundary.

The question is: which path to take? We argue there is a choice so that the homotopy class of the new curve is the same as that of $h_n(\alpha)$. To this end, orient the boundary circle so that there is a ``left" and a ``right" path. Arbitrarily choose one of the two paths connecting the entrance and the exit point, say, $\beta_1$, and consider the loop in the cylinder obtained by concatenating this path with the piece of $h_n(\alpha)\cap C_{d/2}(\xi)$ that connects the endpoints. This is a simple closed curve in a cylinder, and hence there are three possible homotopy classes: trivial, non-trivial and left oriented, and non-trivial and right oriented. If the class is trivial, then using this path $\beta_1$ will work. If it is non-trivial and left oriented, then $\beta_1$ must be the right oriented path. Thus, if we use the opposite path $\beta_2$, it will cancel the homotopy class, and therefore the new path will indeed be homotopic to $h_n(\alpha)$. The right oriented case is similar. 

Each replacement adds at most $\ell(\partial C_{d/2}(\xi))$ to the length of the path. Thus, the length of the new path is (quite crudely) bounded above by $$\ell_{d/2}(h_n(\alpha))+ K\ell(\partial C_{d/2}(\xi)).$$ Therefore, $$\ell_{d/2}(h_n(\alpha))\geq \ell( \psi_n^*\rho_1(\gamma))- K\ell(\partial C_{d/2}(\xi))$$ which tends to $\infty$ as $n\to \infty$, and moreover establishes the result of the lemma.

We are left to consider the case of an unbounded number of crossings into $C_{d/2}(\xi)$ as we take $n\to\infty$. The idea is that each time $h_n(\alpha)$ crosses $\partial C_{d/2}(\xi)$, there is a corresponding ``down-crossing," a curve that connects the point at which it exits to the new point of entry. If each down crossing is denoted by $c_j^n$, then $$\int_{A_x^{d/2}} \Big | \frac{\partial h_n(x,y}{\partial y}\Big | dy \geq \sum_j \ell(c_j^n).$$ Clearly, if the limit of these sums is infinite, then we are done. Hence, we assume the total length due to down-crossings is finite. This implies that there is an even integer $K>0$ such that, as we take $n\to\infty$, there are at most $K$ down-crossing that exit $C_d$. Again using our chosen basepoint, let $b_1$ be the first crossing into $C_{d/2}(\xi)$, and then let $b_2$ be the first crossing out of $C_{d/2}(\xi)$ that exits $C_d$. Set $b_3$ to be the next entry point into $C_{d/2}(\xi)$, and $b_4$ the next exit point from $C_d(\xi)$. We end up with $K_n\leq K$ points $b_1,\dots b_{K_n}$. Replace the segments of $h_n(\alpha)$ between $b_i$ and $b_{i+1}$ with the correct arc on $\partial C_{d/2}(\xi)$ that does not change the homotopy class, where $i=1,3,\dots, K_n-1\leq K-1$. As above, we end up with a curve of length at most $$\ell_{d/2}(h_n(\alpha))+ K\ell(\partial C_{d/2}(\xi))$$ and homotopic to $\ell(\psi_n^*\rho_1(\gamma))$. Using the minimizing property of geodesics once again, we have $$\ell_{d/2}(h_n(\alpha))\geq \ell(\psi_n^*\rho_1(\gamma))- K\ell(\partial C_{d/2}(\xi))\to \infty,$$ and the resolution of this final case completes the proof.
\end{proof}
Returning to (\ref{26}), this lemma shows $\mathcal{E}(\mu_n)\to \infty$, which is a contradiction. We can now conclude that each sequence $\ell(\psi_n^*\rho_1(\gamma))$ remains bounded above.

It is well-understood that the boundedness of the length spectrum implies that $(\psi_n^*\rho_1)$ converges along a subsequence to some new Fuchsian representation $\rho_\infty: \Gamma \to \textrm{PSL}_2(\mathbb{R})$. We now restrict the $[\mu_n], [\psi_n]$ to this chosen subsequence. Since the mapping class group acts properly discontinuously on the Teichm{\"u}ller space, it follows that $[\psi_n^*\rho_1]=[\rho_\infty]$ for all $n$ large enough. In particular, $\psi_n$ is isotopic to $\psi_m$ for all $n,m$ large enough. This completes the proof of Lemma \ref{evenconstant}.

To finish the proof of Proposition \ref{proper}, we know there is an $N$ such that $[\psi_n]=[\psi_N]$ for all $n\geq N$. As the mapping class group acts by isometries with respect to the metric $d_{\mathcal{T}}$, $$d_{\mathcal{T}}([\mu_n],[(\psi_N^{-1})^*\mu_\infty])=d_{\mathcal{T}}([\psi_N^*\mu_n],[\mu_\infty])\to 0,$$ and hence $[\mu_n]$ converges to $[(\psi_N^{-1})^*\mu_\infty]$ as $n\to\infty$. Proposition \ref{proper} is proved.
\end{subsection}
\begin{subsection}{Uniqueness of critical points}
From now on we refine our notation: when the representation $\rho$ is not implicit, the $\rho$-equivariant harmonic map for a metric $\mu$ will now be denoted $f_\mu^{\rho}$. Similarly, if we consider maps $(\tilde{\Sigma},\tilde{\mu})\to (X,\nu)$, we write $e(\mu,\nu,\cdot)$ for an energy density when necessary, rather than $e(\mu,\cdot)$.

We move to the uniqueness statement in Theorem A, which amounts to showing critical points of $\E$ are unique. We can argue similarly to \cite[Theorem 1]{T}. Suppose $\rho_1$ represents a point in $\mathcal{T}_{\mathbf{c}}(\Gamma)$ and almost strictly dominates $\rho_2$, and let $\mu_1$ be a hyperbolic metric representing a critical point in $\mathcal{T}(\Gamma)$ for $\E=\E_{\rho_1,\rho_2}$. Let $\mu_0$ be another hyperbolic metric, representing the point in the Teichm{\"u}ller space $\mathcal{T}_{\mathbf{c}}(\Gamma)$ associated to $\rho_1$ and chosen so that the identity map $\textrm{Id}: (\tilde{\Sigma},\tilde{\mu}_1)\to (\tilde{\Sigma},\tilde{\mu}_0)$ is harmonic. This can be achieved by replacing the target metric with its pullback by the harmonic map, which represents the same point in $\mathcal{T}_{\mathbf{c}}(\Gamma)$. Apriori this is the holonomy of the convex core of a hyperbolic surface, so the universal cover is a proper subset of $\tilde{\Sigma}$ and the identity map only goes into that subset, but we can extend to a global metric by attaching a hyperbolic half-space. Let $\mu_2$ represent yet another point in $\mathcal{T}(\Gamma)$, such that $[\mu_1]\neq [\mu_2]$ in Teichm{\"u}ller space and $\textrm{Id}:(\tilde{\Sigma},\tilde{\mu}_2)\to(\tilde{\Sigma},\tilde{\mu}_0)$ is harmonic--we similarly consider $\textrm{Id}$ as a map to the preimage of the convex core. Also let $\rho$ be the representation associated to $\mu_2$, so that $f_{\mu_1}^\rho:(\Sigma,\mu_1)\to(\Sigma,\mu_2)$ is the harmonic map. We then have decompositions of the pullback metrics
\begin{itemize}
    \item $\mu_0=e(\mu_1,\mu_0,\textrm{Id})\mu_1 + \Phi(f_{\mu_1}^{\rho_2}) + \overline{\Phi}(f_{\mu_1}^{\rho_2}),$
    \item $(f_{\mu_1}^{\rho_2})^*\nu=e(\mu_1,\nu,f_{\mu_1}^{\rho_2})\mu_1+\Phi(f_{\mu_1}^{\rho_2})+\overline{\Phi}(f_{\mu_1}^{\rho_2})$,
    \item $\mu_2=e(\mu_1,\mu_2,f_{\mu_1}^\rho)\mu_1+\Phi(f_{\mu_1}^\rho) + \overline{\Phi}(f_{\mu_1}^\rho)$.
\end{itemize}
In the proof of \cite[Lemma 2.4]{T}, Tholozan computes a formula in local coordinates for the difference of certain energy densities. The computation did not use compactness, so we can use it directly to get
\begin{equation}\label{21}
    (e(\mu_2,\mu_0,\textrm{Id}) - e(\mu_2,\nu,(f_{\mu_1}^{\rho_2})^*\nu))dA_{\mu_2} = \frac{e(\mu_1,\mu_0,\textrm{Id})- e(\mu_1,\nu,(f_{\mu_1}^{\rho_2})^*\nu)}{\sqrt{1- 4\frac{||\Phi(f_{\mu_1}^\rho)||_{\mu_1}^2}{e(\mu_1,\mu_2,f_{\mu_1}^\rho)^2}}}dA_{\mu_1}.
\end{equation}
The denominator lies in $(0,1]$, tends to $0$ at the cusp, and is $1$ exactly when $\overline{\Phi}(f_{\mu_1}^\rho) = 0$. Inputting into our integrals yields $$\int_\Sigma e(\mu_2,\mu_0,\textrm{Id}) - e(\mu_2,\nu,(f_{\mu_1}^{\rho_2})^*\nu)dA_{\mu_2} \geq \E(\mu_1)$$ By Lemma \ref{min}, we have $$\int_\Sigma e(\mu_2,\nu,(f_{\mu_2}^{\rho_2})^*\nu)- e(\mu_2,\nu,(f_{\mu_1}^{\rho_2})^*\nu) dA_{\mu_2} \leq 0,$$ and hence $\E(\mu_2)\geq \E(\mu_1)$, and from (\ref{21}) we have equality if and only if $[\mu_1]=[\mu_2]$ in $\mathcal{T}(\Gamma)$. This establishes the result. As previously discussed, the proof of Theorem A is complete.
\end{subsection}
\end{section}

\begin{section}{Deformations of maximal surfaces}
\begin{subsection}{The maps $\Psi^\theta$}
We now begin the proof of Theorem B. Generalizing the map $\Psi$ from \cite[subsection 2.3]{T}), we define the map $$\Psi^\theta: \mathcal{T}(\Gamma)\times \textrm{Rep}_{\mathbf{c}}(\Gamma,G)\to \textrm{ASD}_{\mathbf{c}}(\Gamma,G)\subset \mathcal{T}_{\mathbf{c}}(\Gamma)\times \textrm{Rep}_{\mathbf{c}}(\Gamma,G)$$ as follows. When $\mathbf{c}$ has no hyperbolic classes, we implicitly assume there is no $\theta$. Begin with a hyperbolic surface $(\Sigma,\mu)$ and a reductive representation $\rho_2:\Gamma\to \PSL(2,\mathbb{R})$. Associated to this representation, we take a $\rho_2$-equivariant harmonic map with twist parameter $\theta$, $f^\theta:(\tilde{\Sigma},\tilde{\mu})\to (X,\nu)$. We proved in \cite[Theorem 1.4]{S} that there exists a unique Fuchsian representation $\rho_1:\Gamma\to \PSL(2,\mathbb{R})$ and a unique $\rho_1$-equivariant harmonic diffeomorphism $h^\theta:(\tilde{\Sigma},\mu)\to(\mathbb{H},\sigma)$ that has Hopf differential $\Phi(h^\theta)=\Phi(f^\theta)$. From \cite[Proposition 3.13]{S}, the mapping $F=(h^\theta,f^\theta)$ is a spacelike maximal immersion into the pseudo-Riemannian product, and hence $\rho_1$ almost strictly dominates $\rho_2$. $\Psi^\theta$ is defined by $$\Psi^\theta([\mu],[\rho_2])=([\rho_1],[\rho_2]).$$ It is clear that the map is fiberwise in the sense of the statement of Theorem B. Theorem A shows $\Psi^\theta$ is bijection, and the content of Theorem B is that $\Psi^\theta$ is a homeomorphism.
\begin{remark}
We do not know if $\Psi^\theta=\Psi^{\theta'}$ for $\theta\neq \theta'$!
\end{remark}
\end{subsection}

\begin{subsection}{Proof of Theorem B: continuity}
 The proof of continuity is almost identical for each twist parameter, so we work only with $\Psi=\Psi^0$. We continue to use the notation from Section 4.3: when $\rho$ is not fixed, the $\rho$-equivariant harmonic map for a metric $\mu$ is denoted $f_\mu^{\rho}.$ Let us also assume for the remainder of this section that there is a single peripheral $\zeta$. Lifting various equivalence relations, $\Psi$ is described by a composition $$(\mu,\rho_2)\mapsto (f_\mu^{\rho_2},\rho_2) \mapsto (\Phi(f_\mu^{\rho_2}),\rho_2) \mapsto (\rho_1,\rho_2)$$ where $\rho_1$ is the holonomy of the hyperbolic metric $\mu'$ on $\Sigma$ such that the associated harmonic map from $ (\Sigma,\mu)\to (\Sigma,\mu')$ has Hopf differential $\Phi(f_\mu^{\rho_2})$. \cite[Theorem 1.4]{S} implies the association from $\Phi(f_\mu^{\rho_2}) \mapsto \rho_1$ is continuous. Here, the topology on the space of holomorphic quadratic differentials is the Fr{\'e}chet topology coming from taking $L^1$-norms over a compact exhaustion. To show continuity of $\mu\mapsto f_\mu^{\rho_2}\mapsto \Phi(f_\mu^{\rho_2})$ with respect to this topology, note that, by the constructions of Section 3, we already have continuity in the Teichm{\"u}ller coordinate. Continuity will thus follow from the results below.

\begin{lem}\label{varyrep}
Suppose representations $(\rho_n)_{n=1}^\infty$ converge to an irreducible $\rho$ in $\textrm{Hom}(\Gamma,G)$, and all such representations project to $\textrm{Rep}_{\mathbf{c}}(\Gamma,G)$. Then $f_\mu^{\rho_n}$ converges to $f_\mu^\rho$ in the $C^\infty$ sense on compacta as $n \to \infty$.
\end{lem}
Below, we work in fixed local sections of the principal bundles $\chi_{\mathbf{c}}(\Gamma,G)\to \textrm{Rep}_{\mathbf{c}}(\Gamma,G)$ $\chi_{\mathbf{c}}(\Gamma,\PSL(2,\mathbb{R}))\to \textrm{Rep}_{\mathbf{c}}(\Gamma,\PSL(2,\mathbb{R}))$ (having assumed in Section 1.2 that such things exist). We choose a basepoint for the $\pi_1$ so that we can view these local systems through their holonomies, which are genuine representations.
\begin{proof}
We may assume each $\rho_n$ is irreducible and that there is a single cusp. The parabolic and elliptic cases are much simpler than the hyperbolic case, so we treat only the latter. Set $f_n=f_{\mu}^{\rho_n}$, $f=f_\mu^\rho$ and let $\varphi_n$, $\varphi$ be the approximation maps for $f_\mu^{\rho_n}$, $f_\mu^\rho$ from Section 3.1. These maps project a cusp neighbourhood onto a geodesic, but now the geodesic is varying with $\rho_n$, and converging in the Gromov-Hausdorff sense to the geodesic for $\rho$. If the $\varphi_n$'s can be chosen to vary smoothly, then we have a uniform bound on the total energies (recall they depend on a choice of a constant speed map onto their geodesic). As discussed in Section 3 (in particular the proof of Lemma \ref{vary}), this yields a uniform total energy bound for $f_n$ on compacta. Via the argument described in Section 3.1, the maps $f_n$ $C^\infty$-converge on compacta along a subsequence to some limiting map $f^\infty$. The mapping is necessarily $\rho$-equivariant and harmonic. To see that $f^\infty=f$, recall that $f$ is characterized by the fact that its Hopf differential has a pole of order $2$ at the cusp and the residue has a specific complex argument. To check this property, from the uniqueness in Theorem \ref{S} it suffices to check $d(f_\infty,\varphi)<\infty$, and the standard contradiction argument will also show there is no need to pass to a subsequence. As remarked previously, the proof of \cite[Lemma 5.2]{S} shows $d(f_n,\varphi_n)$ is maximized on $\partial\Sigma_2$. If $n_k$ is the subsequence, then taking $k\to\infty$, via compactness we do win $$d(f_\infty,\varphi_\infty)\leq \limsup_{k\to\infty} \max_{\partial\Sigma_2}d(f_{n_k},\varphi_{n_k}) \leq  \max_{\partial\Sigma_2}d(f_\infty,\varphi_\infty) + 1< \infty.$$
We are left to argue that one can choose the $\varphi_n$ so that $\varphi_n\to \varphi$ in the $C^\infty$ sense. We realize $\rho,\rho_n$ as monodromies of flat connections $\nabla,\nabla_n$ respectively on an $X$-bundle $E$ over $\Sigma$ with structure group $G$. This is possible when all $\rho_n$ lie in the same component of the relative representation space (\cite[Corollay 6.1.2]{La} generalizes to relative representation spaces), which we are free to assume. The pullback bundle with respect to the universal covering $\tilde{\Sigma}\to \Sigma$ identifies with the trivial bundle and also pulls $\nabla,\nabla_n$ back to flat connections $\tilde{\nabla},\tilde{\nabla}_n$. That is, up to an isomorphism, we have a family of commutative diagrams 
$$\begin{tikzcd}
(X\times \tilde{\Sigma},\tilde{\nabla}_n) \arrow[r]\arrow[d] & (E,\nabla_n) \arrow[d] \\
\tilde{\Sigma} \arrow[r] & \Sigma.
\end{tikzcd}$$
The bundle $E$ and connections $\nabla_n$ are constructed by choosing a good covering of $\Sigma$ and depend on the local section of the principal bundle $\chi_{\mathbf{c}}(\Gamma,G)\to \textrm{Rep}_{\mathbf{c}}(\Gamma,G)$ (see the proof of \cite[Lemma 6.1.1]{La}). We can build a good covering by glueing a good covering of a relatively compact tubular neighbourhood of $(\Sigma_2,\mu)$ with a good covering of a tubular neighbourhood of $(\Sigma\backslash \Sigma_2,\mu)$. This ensures a lower bound on the $\mu$-radius inside $\Sigma_2$ (note that we cannot do this on all of $\Sigma$). With this constraint, the local systems can be prescribed so that $\nabla_n\to\nabla$ in the $C^\infty$ sense on $\Sigma_2$ in the affine space of connections.

The map $\varphi$ induces a section $s$ of the bundle $(E,\nabla)$, which can also be seen as a section $s_n$ of $(E,\nabla_n)$. With respect to the diagram above, $s_n$ pulls back to a $\rho_n$-equivariant map $\tilde{\varphi}_n:\mathbb{H}\to X$, and by our comments above, $\tilde{\varphi}_n\to \varphi$ in the $C^\infty$ sense. The maps $\tilde{\varphi}_n$ are most likely not harmonic and may not project onto the geodesic axis of $\rho_n(\zeta)$. However, because they are converging to $\varphi$, the geodesic curvature of $\tilde{\varphi}_n(\partial\tilde{\Sigma}_2)$ is tending to $0$. Moreover, $\varphi|_{\partial\tilde{\Sigma}_2}$ is arbitrarily close to some parametrization of a geodesic. So, we set $\varphi_n$ to be the harmonic map with boundary data equal to this nearby geodesic projection. From energy control on $\tilde{\varphi}_n$, we get the same for $\varphi_n$, and hence convergence along a subsequence to a harmonic map $\varphi_\infty$. From the boundary values, we must have $\varphi_\infty=\varphi$, and as usual we can show there is no need to pass to a subsequence.
\end{proof}
For a reducible representation, the harmonic maps are not unique but differ by translations along a geodesic. However, the Hopf differential is unique. We can choose a sequence $\varphi_n\to\varphi$ as in the previous proposition---boundary values are fixed so we do have uniqueness for these harmonic maps. Then we take the harmonic maps $f_n$ built from using the approximating maps $\varphi_n$, and the proof above goes through, with the minor modification that we must rescale the approximating maps as in \cite[Proposition 5.1]{S}. $C^\infty$ convergence implies the Hopf differentials converge.
\begin{remark}
For different twist parameters, we simply precompose every $\varphi_n$ with a fractional Dehn twist. Using (\ref{fdehn}), the energies are uniformly controlled.
\end{remark}

\end{subsection}

\begin{subsection}{Asymmetric metrics on Teichm{\"u}ller spaces} In \cite[Section 8]{GK}, Gu{\'e}ritaud-Kassel define asymmetric pseudo-metrics on deformation spaces of geometrically finite hyperbolic manifolds. These metrics are natural generalizations of Thurston's asymmetric metric on Teichm{\"u}ller space. We will use such a metric in the proof of bi-continuity of $\Psi^\theta$.

For $[\rho_1],[\rho_2]\in \mathcal{T}_\mathfrak{c}(S_{g,n})$, we set 
$$C(\rho_1,\rho_2) = \inf \{\textrm{Lip}(f) : f:\mathbb{H}\to \mathbb{H} \textrm{ is }(\rho_1,\rho_2)-\textrm{equivariant} \}.$$
The metric $d_{Th}: \mathcal{T}_\mathfrak{c}(S_{g,n})\times \mathcal{T}_\mathfrak{c}(S_{g,n})\to\mathbb{R}$ is defined by $$d_{Th}([\rho_1],[\rho_2]) = \log \Big ( C(\rho_1,\rho_2) \frac{\delta(\rho_1)}{\delta(\rho_2)}\Big ).$$ Here $\delta(\rho_1)$ is the critical exponent of $\rho_1$: $$\delta(\rho_1) = \lim_{R\to\infty} \frac{1}{R}\log \#(j(\Gamma)\cdot p \cap B_p(R)),$$ where $p$ is any point in $\mathbb{H}$ and $B_p(R)$ is the ball of radius $r$ centered at $p$. Alternatively, $\delta(\rho_1)$ is the Hausdorff dimension of the limit set of $\rho_1$ in $\partial_\infty\mathbb{H}=S^1$. If all $a_j=0$, so that all critical exponents are $1$ and $\mathcal{T}_\mathfrak{c}(S_{g,n})$ is the ordinary Teichm{\"u}ller space of a surface of finite volume, then this agrees with Thurston's asymmetric metric introduced in \cite{Th}. Here asymmetric means that in general, $$d_{Th}([\rho_1],[\rho_2])\neq d_{Th}([\rho_2],[\rho_1]).$$
\begin{prop}[Gu{\'e}ritaud-Kassel, Lemma 8.1 in \cite{GK}]
The function $d_{Th}:\mathcal{T}_{\mathbf{c}}(\Gamma)\times \mathcal{T}_{\mathbf{c}}(\Gamma)\to \mathbb{R}$ is a continuous asymmetric metric.
\end{prop}
\begin{remark}
The correction factor $\delta(\rho_1)/\delta(\rho_2)$ is needed for non-negativity. For a general hyperbolic $n$-manifold, continuity may fail and the generalization may be just an asymmetric pseudo-metric ($d(x,y)$ need not imply $x=y$).
\end{remark}
This metric allows us to control translation lengths nicely. Consider a left open ball around $\rho_2$: $$B_{Th}(\rho_2,C)= \{[\rho_1]\in\mathcal{T}_{\mathbf{c}}(\Gamma): d_{Th}(\rho_1,\rho_2)<C\}.$$ 
If $\rho_1\in B_{Th}(\rho_2,C)$, there is a $(\rho_1,\rho_2)$-equivariant Lipschitz $g$ map with Lipschitz constant $<e^C$. Therefore, for every $\gamma\in \Gamma$, $$d(\rho_2(\gamma)g(z),g(z)) = d(g(\rho_1(\gamma)z), g(z)) < e^C d(\rho_1(\gamma)z,z),$$ and it follows that $\ell(\rho_2(\gamma))< e^C\ell(\rho_1(\gamma))$.

\end{subsection}

\begin{subsection}{Proof of Theorem B: bi-continuity}
The inverse mapping of $\Psi^\theta$, on the Teichm{\"u}ller side, takes as input a class of an almost strictly dominating pair $(\rho_1,\rho_2)$ and returns the unique minimizer of $\E_{\rho_1,\rho_2}^\theta$. We can follow the same approach as in \cite[subsection 2.4]{T} here, adapted to our infinite energy setting. 

Let $X,Y$ be metric spaces and $(F_y)_{y\in Y}$ a family of continuous
functions $F_y: X\to \mathbb{R}$ depending continuously on $y$ in the compact-open topology. $(F_y)_{y\in Y}$ is said to be uniformly proper if for any $C\in\mathbb{R}$, there exists a compact subset $C\subset X$ such that for all $y\in Y$ and $x\not \in K$, we have $F_y(x) > C$. We say that the family $(F_y)_{y\in Y}$ is locally uniformly proper if for all $y_0 \in Y$, there is a
neighbourhood $U$ of $y_0$ such that $(F_y)_{y\in U}\subset (F_y)_{y\in Y}$ is uniformly proper.
\begin{lem}[Proposition 2.6 in \cite{T}]\label{unifproper}
Let $X$ and $Y$ be two metric spaces and $(F_y)_{y\in Y}$ a locally uniformly proper family of continous functions from $X$ to $\mathbb{R}$ depending continuously on $Y$ (for the compact open topology). Assume that each $F_y$ achieves its minimum at a unique point $x_m(y)\in X$. Then the function $$y\mapsto x_m(y)$$ is continuous.
\end{lem}
We verify the conditions for $\E_{\rho_1,\rho_2}^\theta$, with $(\rho_1,\rho_2)$ living in the space of almost strictly dominating representations. For the remainder of this subsection, we are working over local sections for our bundles of local systems over representation spaces.
\begin{lem}
The association $(\rho_1,\rho_2)\mapsto \E_{\rho_1,\rho_2}^\theta$ is continuous for the compact-open topology.
\end{lem}
In the proof, we require control on the energy of the harmonic maps at the cusp, as we vary the source metric and the representation. We defer the proof of the lemmas below to the next subsection. In these lemmas, let $\rho_0$ be a reductive representation and $\mu_0$ a hyperbolic metric.
\begin{lem}\label{varcusp1}
Suppose $\rho_0$ has parabolic monodromy. Then for every representation $\rho$ in the same representation space that is close enough to $\rho_0$, and metric $\mu$ close to $\mu_0$, there is a function $\tilde{e}$ that is integrable in the flat metric and such that $$\sqrt{|\mu|}e(\mu,f_\mu^\rho)\leq \tilde{e}$$ everywhere.
\end{lem}
\begin{lem}\label{varcusp2}
Suppose $\rho_0$ has hyperbolic monodromy. Then for every representation $\rho$ in the same representation variety that is close enough to $\rho_0$, and metric $\mu$ close to $\mu_0$, working in the cusp coordinates for $\mu$ there is a $y_0>0$, $C,c>0$ such that for all $y\geq y_0$, $$\frac{\Lambda(\theta)\ell^2}{2\tau_\mu^2}- Ce^{-cy}\leq \sqrt{\mu}e(\mu,f_\mu^\rho)\leq \frac{\Lambda(\theta)\ell^2}{2\tau_\mu^2}+ Ce^{-cy},$$ where $f_\mu^\rho$ is the harmonic map with twist parameter $\theta$.
\end{lem}

\begin{proof}
We want to show that if $(i_n,j_n)\to (\rho_1,\rho_2)$ in $\textrm{Rep}_{\mathbf{c}}(\Gamma,\PSL(2,\mathbb{R})\times G)$ and $K\subset \mathcal{T}(\Gamma)$ is compact, then $\E_n :=\E_{i_n,j_n}^\theta\to \E=\E_{\rho_1,\rho_2}^\theta$ uniformly on $K$ as $n\to \infty$. The finite energy case is easy: from our previous results, if $\mu_n\to \mu$ and $\rho_n\to \rho$, then the energy densities of the harmonic maps converge pointwise to $e(\mu,f_\mu^\rho)$ (recall this is independent of the harmonic map if $\rho$ is reducible). Lemma \ref{varcusp1} then justifies an application of the domination convergence theorem, so that the total energies converge.

Going forward, we assume the monodromoy is hyperbolic. Fixing a metric $\mu$, set $h_n=h_\mu^{i_n}$, $f_n=f_\mu^{j_n}$, $h=h_\mu^{\rho_1}$, $f=f_\mu^{\rho_2}$. It suffices to show $\E_n(\mu)\to \E(\mu)$, and that the rate only depends on the Teichm{\"u}ller distance from $\mu$ to a base metric $\mu_0$. By Lemma \ref{varcusp2}, for $r>y_0$,
\begin{align*}
    |\E(\mu) - \E_n(\mu)| &= \Big |\int_{\Sigma_r} (e(\mu,h) - e(\mu,f)) - (e(\mu,h_n)-e(\mu,f_n)) dA_\mu\Big |  \\
    &+ \Big|\int_{\Sigma\backslash\Sigma_r}(e(\mu,h) - e(\mu,f)) - (e(\mu,h_n)-e(\mu,f_n)) dA_\mu \Big |\\
    &\leq \Big |\int_{\Sigma_r} (e(\mu,h) - e(\mu,f)) - (e(\mu,h_n)-e(\mu,f_n)) dA_\mu\Big | + \int_{\Sigma\backslash\Sigma_r} Ce^{-cy} dA_\mu
\end{align*}
holds for every $\mu$ close enough to $\mu_0$. Fixing $\epsilon>0$, for every $r\geq y_0$ we can find $N_r> 0$, depending only on Teichm{\"u}ller distance to $\mu_0$, such that for all $n\geq N_r$, the first integral is $<\epsilon/2$. Hence, for such $n$, $$|\E(\mu) -\E_n(\mu)|<\epsilon/2 + \frac{2\pi C}{c}e^{-cr}.$$ Taking $r=c^{-1}\log(4\pi C\epsilon^{-1})$, we get $|\E(\mu) -\E_n(\mu)|<\epsilon$.
\end{proof}
We now show that the functionals $\E_{\rho_1,\rho_2}^\theta$ are locally uniformly proper. From here we assume $\theta=0$, because the proof is identical for every $\theta$. We essentially show that the bounds from the proof of Proposition \ref{proper} depend continuously on $([\rho_1],[\rho_2])$. For $([\rho_1],[\rho_2])\in \textrm{ASD}_{\mathbf{c}}(\Gamma,G)$, we choose an open neighbourhood $U\subset \textrm{ASD}_{\mathbf{c}}(\Gamma,G)$ containing $([\rho_1],[\rho_2])$ with compact closure. We intersect it with a product open set $U_1\times U_2$, where $U_1$ is a left $d_{Th}$-open ball around $\rho_1$. We then lift via some section to the space of local sections to view these points as representations. Picking a boundary geodesic or a horocycle for $C(\mathbb{H}/\rho_1(\Gamma))$, as we perturb the representations we get a continuously varying family of such curves in the new metric. We write $C_d^j(\xi)$ for the collar neighbourhood of such a curve in $\mathbb{H}/j(\Gamma)$, $j\in U_2$. By choosing $U$ even small enough, we can assume we have a fixed presentation for our fundamental group, and the collar neighbourhood $C_{d_j}^j(\xi)$ has uniform upper and lower bounds $\delta_1 \leq d_j\leq \delta_2$. 

Set $\delta=(\delta_1+\delta_2)/2$. We can choose a neighbourhood $C_\delta$ containing every $C_{\delta}^j(\xi)$ for all $j\in U_2$. For a $(j,\rho)$-equivariant map $g$, we put $$\textrm{Lip}_\delta(g) = \max_{x\in C(\mathbb{H}/j(\Gamma))\backslash C_\delta} \textrm{Lip}_x(g).$$
\begin{lem}
Shrinking $U$ if necessary, there exists an $\epsilon>0$ such that for every $(j,\rho)\in U$, there is a $(j,\rho)$-equivariant map $g_{j,\rho}$ that satisfies $\textrm{Lip}(g_{j,\rho})\leq 1$ and $\textrm{Lip}_\delta(g_{j,\rho})<(1-\epsilon)^{1/2}$. Moreover, we can choose it so that if $\rho(\zeta)$ is hyperbolic, then $g_{j,\rho}$ translates the geodesic axis of $j(\zeta)$ along the axis of $\rho(\zeta)$ with constant speed $1$.
\end{lem}
\begin{proof}
We first define $g_{j,\rho}$ on the complement of $C_\delta$. One at a time, we vary $j$ and then $\rho$. By our choice of $U_1$, there is an $\epsilon_0>0$ such that for every $j\in U_1$, there is a $(j,\rho_1)$-equivariant $(1+\epsilon_0)$-Lipschitz map. Composing with our original optimal $(\rho_1,\rho_2)$-equivariant map, we get a $(j,\rho_2)$-equivariant map with nice control. Choosing $U_1$ small enough, we can shrink $\epsilon_0$ so as to ensure the right behaviour outside of $C_\delta$. 

Now we fix a base surface $\mathbb{H}/j(\Gamma)$ and vary $\rho$ around $\rho_2$. We can use flat connections as in Lemma \ref{varyrep}. For any continuous path of classes with initial point $\rho_2$, the procedure detailed there gives a path of equivariant maps starting at $g$. From compactness of the complement of the collar and cusp neighbourhoods, the local Lipschitz constants vary upper semicontinuously. In particular, we can achieve an upper bound $\textrm{Lip}_\delta(\cdot)<(1-\epsilon)^{1/2}$ when close enough to $([\rho_1],[\rho_2])$. 

Now we extend in $C_\delta$ and above. Note that while the local Lipschitz constants of $g_{j,\rho}$ are uniformly controlled, this can be strictly below the global Lipschitz constant, and in the hyperbolic case this global Lipschitz constant is exactly $1$. To see this, we do have $d_\nu(g_{j,\rho}(x),g_{j,\rho}(y))<d_\sigma(x,y)$ for every $x\neq y$, so $\textrm{Lip}(g_{j,\rho})\leq 1$ certainly. But in the case of hyperbolic monodromy, for any two points $x,y$ that are connected by a segment that mostly fellow-travels the geodesic axis of $j(\zeta)$, $$ d_\nu(g_{j,\rho}(x),g_{j,\rho}(y))= d_\sigma(x,y)+O(1),$$ where the implied constant depends only on the position of $x$ and $y$ in $\mathbb{H}/j(\Gamma)$. Since $x,y$ can be taken as far as we like, in the limit the ratio of distances becomes $1$. Using the equivariant Kirszbraun-Valentine theorem \cite[Proposition 3.9]{GK}, adapted to the stabilizer of the cusp or funnel, we extend each such equivariant map to a globally defined equivariant map with global Lipschitz constant $\leq 1$ in the parabolic and elliptic cases, and exactly $1$ in the hyperbolic case. The constraint $\ell(j(\zeta))=\ell(\rho(\zeta))$ forces $g_{j,\rho}$ to translate along the geodesic.
\end{proof}
\begin{remark}
\cite[Proposition 3.9]{GK} is only proved for equivariant maps from hyperbolic $n$-space to itself. However, a version still holds for maps from $(\mathbb{H},\sigma)$ to any $\textrm{CAT}(-1)$ metric space. The proof involves taking barycenters of Lipschitz maps, which can be done just the same in any $\textrm{CAT}(0)$ space, and a few applications of the Toponogov theorem that go through in a $\textrm{CAT}(-1)$ setting.
\end{remark}
\begin{remark}
In \cite[Proposition 3.9]{GK}, only Lipschitz constant at least $1$ is addressed. Using compactness of $C_{\delta}$ and adapting the proof using \cite[Proposition 3.7]{GK} rather than Propositions 3.1 and Remark 3.6, we acquire the result in this other context (with a potential loss on the Lipschitz constant).
\end{remark}

Returning to the main proof, we write $h$ to be a Fuchsian harmonic map, omitting dependence on the metric and representation. For any $(j,\rho)\in U$, the fact that $g_{j,\rho}$ translates along the geodesic means we can apply Lemma \ref{min}: $$\E_{j,\rho}(\mu) \geq \int_{\Sigma} e(\mu, h) - e(\mu, g_{j,\rho}\circ h) dA_\mu.$$ And using that $\textrm{Lip}(g)=1$, we get $$\E_{j,\rho}(\mu) \geq \int_{K} e(\mu, h) - e(\mu, g_{j,\rho}\circ h) dA_\mu$$ for any compact $K\subset \Sigma$.

Fix a simple closed curve $\gamma\in \Gamma$. We are positioned to repeat the initial computation in the proof of Theorem A, and doing so gives that for $([j],[\rho])\in U$ and $[\mu]\in\mathcal{T}(\Gamma)$ we have 
\begin{equation}\label{22}
    \E_{j,\rho}(\mu)\geq \frac{w_{\mu}}{2}\epsilon\min_x\Big (\int_{A_x^t}  \Big |\frac{\partial h(x,y)}{\partial y}\Big | dy \Big )^2,
\end{equation}
 where $A_x^t$ is defined as in Section 4.2, and $w_\mu$ is the $\mu$-length of the collar associated to $\gamma$. Repeating the proof of Lemma \ref{nocollar}, almost word for word, we can see 
$$\int_{A_{x}^t}  \Big |\frac{\partial h(x,y)}{\partial y}\Big | dy\geq \textrm{min}\{\delta_1/2,\kappa_\gamma\},$$ where $\kappa_\gamma$ is the minimum of the lengths for $j(\gamma)$. If the $\mu$-length of any $\gamma$ goes to $0$, then this integral explodes. Thus, for any $([j],[\rho])\in U$, and curve $\gamma$, there is a length $\epsilon_\gamma$ such that if $\ell_\sigma(\gamma)<\epsilon_\gamma$, the right-hand-side of (\ref{22}) is greater than $K$. This implies there is a compact subset of the moduli space such that if we take a fundamental domain $V$ for this subset in Teichm{\"u}ller space, then we have the $\E_{j,\rho} >K$ on the complement of the mapping class group orbit of $V$. 

To finish the proof, we show there are only finitely many mapping classes $[\psi]$ such that $\E_{j,\rho}\leq K$ on the translate $\psi^*V$. Suppose there exists a metric $\mu$ representing a point in $V$ and a sequence of distinct mapping classes $\psi_n$ such that $\E_{j,\rho}(\psi_n^*\mu) \leq K$ for all $([j],[\rho])\in U$. Then, by proper discontinuity of the mapping class group, $(\psi_n^{-1})^*j$ diverges in Teichm{\"u}ller space for every $j$. This implies that for each $j$ there exists a non-trivial simple closed curve $\gamma_j$ whose length under $(\psi_n^{-1})^*j$ blows up as $n\to \infty$ .

Since we intersected with a left open ball for $d_{Th}$, we can choose all $\gamma_j$ to be equal to a single curve $\gamma$.  If $C$ is the radius for our left open ball, then for all $n$ and $\gamma\in \Gamma$,
\begin{equation}\label{23}
    \ell((\psi_n^{-1})^*j(\gamma))\geq e^{-C}\ell(\psi_n^{-1})^*\rho_1(\gamma)).
\end{equation}
Thus, if $\gamma\in \Gamma$ is such that $\ell((\psi_n^{-1})^*\rho_1(\gamma))$ blows up, then by (\ref{23}), the same holds for every $j$ sufficiently close by. Moreover, the rate at which $\ell((\psi_n^{-1})^*j(\gamma))\to \infty$ is independent of $j$, close to that of $(\psi_n^{-1})^*\rho_1(\gamma)$. Now we have an integral estimate as in Theorem A: $$K\geq \E_{j,\rho}(\psi_n^*\mu)=\E_{(\psi_n^{-1})^*j ,(\psi_n^{-1})^*\rho}(\mu) \geq  \frac{w}{2}\epsilon\min_x\Big (\int_{A_x^t}  \Big |\frac{\partial h(x,y)}{\partial y}\Big | dy \Big )^2,$$ where $w$ is minimum of the lengths of the collars around $\gamma$ for $[\mu]\in V$. The proof of Lemma \ref{bigcurve} can then be made uniform: by examination of the proof, the integral $$\int_{A_x^t} \Big |\frac{\partial h(x,y)}{\partial y}\Big | dy$$ trails off to infinity with a rate depending on that of the translation length of the bad sequence. This is a contradiction, and thus the energy functional does have the $>K$ condition on the complement of a finite orbit. Therefore, we've satisfied Lemma \ref{unifproper}, and modulo Lemmas \ref{varcusp1} and \ref{varcusp2}, finished the proof of Theorem B.
\end{subsection}
\begin{subsection}{Variations at the cusp} Here we prove Lemmas \ref{varcusp1} and \ref{varcusp2}. Let $\mu$ be any metric close to $\mu_0$. Uniformizing a neighbourhood of the cusp to a punctured disk, we consider the Hopf differential as a meromorphic function for the complex structure of $\mu$ with a pole of order at most two: $$\phi(z) = -\frac{\Lambda(\theta)\ell^2}{16\pi^2}z^{-2} + a_\mu z_\mu^{-1} + \varphi_\mu(z),$$ where $z_\mu$ is a holomorphic coordinate for $\mu$, and $\ell$ is the translation length of the peripheral curve in question. If the representation does not have hyperbolic monodromy at the cusp, then it is understood that $\ell=0$ in the expression above. We can choose a neighbourhood of the puncture containing cusp neighbourhoods for all $\mu$ that uniformize for $\mu$ to an open set containing a punctured disk of $\mu$-radius uniformly bounded below.

It follows from the results of Section 3 and the proof of continuity in Theorem B that for any $(\mu_n,\rho_n)$ converging to $(\mu,\rho)$, the harmonic maps can be chosen, even in the reducible case, to converge in the $C^\infty$ sense on compacta. This implies the Hopf differentials, viewed simply as smooth rather than holomorphic functions, converge to that of $f_{\mu_0}^{\rho_0}$ locally uniformly on compacta in this punctured disk. If $z=z_{\mu_0}$, then after choosing our normalizations correctly, $z_{\mu}\to z$ as $\mu \to \mu_0$. It follows that the Laurent coefficients converge, and hence $a_\mu$ is bounded and $\varphi_\mu$ is $C^0$ bounded.

To prove Lemma \ref{varcusp1}, we first assume $\rho$ is Fuchsian. Then, in the coordinates on $\mathbb{D}^*$, the Beltrami form $\psi$ satisfies $$|\psi| = \frac{|\Phi|}{\mu H(\mu,f_\mu)}\leq  \frac{|\Phi|}{\mu}\leq \frac{Cz^{-1}}{|z|^{-2}(\log|z|)^2}=C|z|(\log|z|)^{-1}\to 0$$ as $z\to 0$. Via this decay on the Beltrami form, we know that once we go high enough into the cusp, $f$ there is a uniform bound on the quasiconformal dilatation (independent of $\mu$ and $\rho$). Thus from uniform convergence on compacta, we have a uniform $K$-quasiconformal bound everywhere. By the Schwarz lemma for quasiconformal harmonic maps \cite{GH}, we extract the bound $H(\mu,f_\mu)\leq 2K^2$. Since $L\leq H$ in the Fuchsian case, $e(\mu,f_\mu)\leq 4K^2$. If $\rho$ is not Fuchsian, then by Proposition 3.13 from \cite{S} we can bound the energy density above by that of the harmonic map for the Fuchsian representation with the same Hopf differential. This constant is integrable, and this proves Lemma \ref{varcusp1}.

Lemma \ref{varcusp2} is a bit more work. Passing to the cusp coordinates, the uniform bound on the Laurent coefficients implies there is a uniform $C,c>0$ such that $\phi$ satisfies 
\begin{equation}\label{14}
    \frac{\Lambda(\theta)\ell^2}{4\tau_\mu^2} + Ce^{-cy} \leq |\phi| \leq \frac{\Lambda(\theta)\ell^2}{4\tau_\mu^2} + Ce^{-cy}
\end{equation}
in $\Sigma\backslash\Sigma_s$. Setting $\psi=\psi_\mu$ to be $|(f_\mu)_{\overline{z}}|/|(f_\mu)_z|$, the formula
\begin{equation}\label{13}
    e(\mu^f,f_\mu)=|\Phi|(|\psi|+|\psi|^{-1})
\end{equation}
 suggests we should turn to $|\psi|$. Using (\ref{14}), we find uniform upper and lower bounds on $|\psi|$, independent of $(\mu,\rho)$. If there are no such bounds, then there is a sequence $\mu_n$ tending to $\mu$ and points $z_n$ with $|\psi_n|(z_n)\to 0$ or $|\psi_n|(z_n)\to \infty$, where $|\psi_n|=L(f_n)^{1/2}/H(f_n)^{1/2}$, for $f_n=f_{\mu_n}$. We can assume each $z_n$ lies in a cylinder of the form $(\Sigma_{r_n}\backslash\Sigma_{r_n-1},\mu)$. Taking $i_n:\mathcal{C}\to (\Sigma_{r_n}\backslash\Sigma_{r_n-1},\mu)$ to be a cylinder embedding, conformal for $\mu$, uniform energy density bounds imply convergence of $F_n=f_n\circ i_n$ along a subsequence to a limiting harmonic map $F_\infty : \mathcal{C}\to (X,\nu)$. From the inequalities (\ref{14}), the Hopf differential is exactly $\Lambda(\theta)\ell^2/4\tau_{\mu_0}^2 dz_{\mu_0}^2$. $F_\infty$ projects onto the geodesic and hence has rank $1$. Thus, there is an $\eta\in\mathbb{R}$ such that $(F_\infty)_*(\partial/\partial x)=\eta (F_\infty)_*(\partial/\partial_y)$. Hence, writing out the Hopf differential in coordinates gives $$\Phi(F_\infty) = \frac{1}{4}(|(F_\infty)_*(\partial/\partial x)|_\nu^2-|(F_\infty)_*(\partial/\partial_y)|_\nu^2 - 2i\langle f_x,f_y\rangle_\nu)=\frac{1}{4}(\eta^2 -1 -2i\eta)dz^2.$$ Therefore, $\eta=\theta$. We thus find from the linear ODE theory that the limit is a constant speed parametrization of the geodesic, composed with a fractional Dehn twist. This implies the limiting quantity $|\psi_\infty|$ is exactly $1$, which contradicts our assumption $|\psi|(z_n)\to 0$ or $\infty$.
 
 From uniform bounds we upgrade to more precise control. Let us temporarily assume $\rho$ is Fuchsian. Working in the region where we have these bounds, because the pullback metric for our harmonic map is hyperbolic, it can be deduced from the Bochner formulae \cite[Chapter 1]{SY} that $$\Delta_{\mu^f} \log |\psi|^{-1}=2|\Phi_\mu|\sinh \log |\psi|^{-1}.$$ Hence, $$\frac{\Lambda(\theta)\ell^2}{4\tau_\mu^2}\log |\psi|^{-1}\leq \Delta_{\mu^f} \log |\psi|^{-1}\leq \frac{\Lambda(\theta)\ell^2}{2\tau_\mu^2}\log |\psi|^{-1}$$ when $|\psi|< 1$, if we are high enough to get uniform control on $|\Phi_\mu|$. If $|\psi|>1$, we have the opposite inequality $$\frac{\Lambda(\theta)\ell^2}{2\tau_\mu^2}\log |\psi|^{-1}\leq \Delta_{\mu^f} \log |\psi|^{-1}\leq \frac{\Lambda(\theta)\ell^2}{4\tau_\mu^2}\log |\psi|^{-1}.$$  Our uniform bounds on $|\psi|$ give control on $\log |\psi|^{-1}$, which yields more bounds of the form $$-\frac{c\Lambda(\theta)\ell^2}{2\tau_\mu^2}\leq \Delta_{\mu^f}\log|\psi|^{-1}\leq \frac{C\Lambda(\theta)\ell^2}{2\tau_\mu^2}.$$ Using the maximum principle, we can then deduce $$-Ce^{-c y} \leq \log |\psi|^{-1}\leq Ce^{-c y}.$$ Taylor expanding $x\mapsto \log(1-x)$, we then obtain $$1-Ce^{-cy} \leq |\psi| \leq 1+Ce^{-cy}.$$ If $\rho$ is not Fuchsian, we apply an argument similar to that of Lemma \ref{asym} to get this same asymptotic. Inserting the bounds into the formula (\ref{13}) gives $$\frac{\Lambda(\theta)\ell^2}{2\tau_\mu^2}-Ce^{-ct}\leq e(\mu^f,f_\mu) \leq \frac{\Lambda(\theta)\ell^2}{2\tau_\mu^2}+Ce^{-ct},$$ as desired. This completes the proof of Lemma \ref{varcusp2}, and moreover the proof of Theorem B.
 \end{subsection}
\end{section}

\begin{section}{Anti-de Sitter 3-manifolds}
After introducing $\AdS$ geometry, we prove Theorem C. 

\begin{subsection}{Models of $\AdS$}
The exposition here is minimal, and for more information we suggest the recent survey \cite{BS}. Denote by $\mathbb{R}^{2,2}$ the real vector space $\mathbb{R}^4$ equipped with the non-degenerate bilinear form $$Q(x,y)=x_1y_1+x_2y_2-x_3y_3-x_4y_4.$$ We define $$\mathbb{H}^{2,1}=\{x\in \mathbb{R}^{2,2}: Q(x,x)=-1\}.$$ The quadric $\mathbb{H}^{2,1}\subset \mathbb{R}^{2,2}$ is a smooth connected submanifold of dimension $3$, and each tangent space $T_x\mathbb{H}^{2,1}$ identifies with the $Q$-orthogonal complement of the linear span of $x$ in $\mathbb{R}^{2,2}$. The restriction of $Q$ to such a tangent space is a non-degenerate bilinear form of signature $(2,1)$, and this induces a Lorentzian metric on $\mathbb{H}^{2,1}$ of constant curvature $-1$ on non-degenerate $2$-planes. $\mathbb{H}^{2,1}$ identifies with the Lorentzian symmetric space $O(2,2)/O(2,1)$, where $O(2,1)$ embeds into $O(2,2)$ as the stabilizer of $(0,0,0,1)$. For our purposes, we mostly work with $SO_0(2,2)$, the space and time orientation preserving component of $O(2,2)$.  

The center of $O(2,2)$ is $\{\pm I\}$, where $I$ is the identity matrix. The Klein model of $\AdS$ is the quotient $$\AdS = \mathbb{H}^{2,1}/\{\pm I\},$$ with the Lorentzian metric induced from $\mathbb{H}^{2,1}$. This also identifies as the space of timelike directions in $\mathbb{R}^{2,2}$, $$\AdS=\{[x]\in\mathbb{RP}^{3}:Q(x,x)<0\}.$$ 
In the introduction, we discussed another model of $\AdS$: the Lie group $\PSL(2,\mathbb{R})$. The determinant form $q=(-\textrm{det})$ defines a signature $(2,1)$ bilinear form on the lie algebra $\mathfrak{sl}(2,\mathbb{R})=T_{[I]}\PSL(2,\mathbb{R})$ (it is a multiple of the Killing form). Translating to each tangent space via the group multiplication, we obtain a Lorentzian metric that is isometric to $\AdS$. The space and time-orientation preserving component of the isometry group is $\PSL(2,\mathbb{R})\times\PSL(2,\mathbb{R})$, acting via the left and right multiplication.

A tangent vector $v\in T_x\mathbb{H}^{2,1}$ is timelike, lightlike, and spacelike if $Q(v,v)<0$, $Q(v,v)=0$, and $Q(v,v)>0$ respectively, and likewise for $\AdS$. The causal character of a geodesic curve is constant, and correspondingly we call geodesics timelike, lightlike, or spacelike if every tangent vector is timelike, lightlike, or spacelike. In the $\PSL(2,\mathbb{R})$ model, timelike geodesics are all of the form $$L_{p,q}=\{X\in\PSL(2,\mathbb{R}): X\cdot p = q\},$$ where $(p,q)$ range over $\mathbb{H}\times \mathbb{H}$. These are topological circles and have Lorentzian length $\pi$.
\end{subsection}

\begin{subsection}{AdS $3$-manifolds with $S^1$-fibrations}
In this subsection we prove Proposition \ref{equiv}, which is actually a quick consequence of the proposition immediately below. We work in the $\PSL(2,\mathbb{R})$ model throughout.
\begin{prop}
Let $V\subset \mathbb{H}$ be a domain. The data of a domain $\Omega\subset \AdS$ and a fibration $\Omega\to V$ such that every fiber is a timelike geodesics is equivalent to that of a domain $V\subset\mathbb{H}$ and a locally strictly contracting map $g:V\to \mathbb{H}.$
\end{prop}
The proof of the first direction of the equivalence is a straightforward adaptation of the procedure from \cite[Proposition 7.2]{GK}. There, $V=\mathbb{H}$, $\Omega=\AdS$, and $h$ is (globally) strictly contracting. We include the proof for the readers convenience.
\begin{proof}
The key fact we use is that timelike geodesics $L_{p_1,q_1}$ and $L_{p_2,q_2}$ intersect if and only if $$d_\sigma(p_1,p_2)=d_\sigma(q_1,q_2).$$ With this in mind, given a locally strictly contracting mapping $g: V\to\mathbb{H}\times \mathbb{H}$ with the properties above, timelike geodesics of the form $L_{p,g(p)}$ and $L_{q,g(q)}$ never intersect. Thus, the geodesics $L_{p,g(p)}$ sweep out a connected set $\Omega\subset \AdS$ as $p$ ranges over $V$. 

We argue that $\Omega$ is open. We record that $X\in L_{p,g(p)}$ if and only if $$X^{-1}\circ g(p) =p.$$ For small $\epsilon>0$, let $B_\epsilon(p)\subset V$ denote the $\epsilon$-ball around $p$ in $\mathbb{H}$. Let $B\subset \AdS$ be the open ball consisting of isometries $Y$ such that $$d_{\sigma}(p,Y^{-1}g(p))< (1-\textrm{Lip}(g|_{B_\epsilon(p)}))\epsilon.$$ Then for any $q\in B_\epsilon(p)$ and $Y\in B$, $$d_{\sigma}(Y^{-1}\circ g(q),p)\leq d(Y^{-1} g(q), Y^{-1}g(p)) + d(Y^{-1} g(p),p)< \epsilon.$$ Thus, $Y^{-1}g$ takes the closure of $B_\epsilon(p)$ to itself, and by the Banach fixed point theorem there is a unique $q\in B_\epsilon(p)$ such that $Y\circ g (q) =q$. So $B\subset \Omega$. This argument also shows that the fibration from $\Omega\to V$ described by $L_{f(p),h(p)}\mapsto p$ is continuous.

For the other direction, any circle fibration $\Omega\to V$ with timelike geodesic fibers determines a map $F:V\to \mathbb{H}\times \mathbb{H}$ by $F(p)=(h(p),f(p))$, where $L_{f(p),h(p)}$ is the geodesic lying over $p$ in $\Omega$. $F$ preserves connectedness---using the product structure, \cite[Theorem 2.2]{JJ} guarantees it is continuous when $f$ is non-constant. If $f$ is a constant $q$, then because $\Omega$ is open, for any $p$ and path from $p$ to $q$, we can find a continuous path of isometries $r\mapsto X_r$ such that $h(r)=X_r^{-1}q$. Thus we have continuity here as well. As the timelike geodesics never intersect, $d(f(p),f(q))\neq d(h(p),h(q))$ for $p\neq q$. As the diagonal in $\mathbb{H}\times \mathbb{H}$ has codimension $2$, a connectedness argument shows $d(f(p),f(q))<d(h(p),h(q))$ or $d(f(p),f(q))>d(h(p),h(q))$ always for $p\neq q$. By switching coordinates, we may assume we have the former. This condition ensures that $h$ is injective. Therefore, $g=f\circ h^{-1}$ is a well-defined locally strictly contracting map.
\end{proof}
Proposition \ref{equiv} is just the equivariant version of this: for a pair $(\rho_1,\rho_2)$ with $\rho_1$ acting properly discontinuously on $V$, we have $$\rho_2(\gamma)L_{p,g(p)}\rho_1(\gamma)^{-1} = L_{\rho_1(\gamma)p,\rho_2(\gamma)g(p)},$$ so $\rho_1\times \rho_2$ acts properly discontinuously on $\Omega$ and equivariance of the fibration is clear. 

It is seen in the proof that $\Omega \subset \AdS$ consists of all isometries $X$ such that $X^{-1}\circ g$ has a fixed point.

\begin{remark}
The results here generalize, almost word for word, for quotients of proper domains in the rank $1$ Lie groups $G=\textrm{O}(n,1)$, $\textrm{SO}(n,1)$, $\textrm{SO}_0(n,1)$, and $\textrm{PO}(n,1)$. One can consider the action by left and right multiplication and equivariant $K$-fibrations $G\supset \Omega\to V\subset\mathbb{H}^n$, $n\geq 2$, where $K\subset G$ is the maximal compact subgroup. Here the fibers are copies of $K$, each of the form $\{X\in G: X\cdot p=q\}$ for some $p,q\in\mathbb{H}^n$.
\end{remark}
\begin{remark}
Proposition \ref{equiv} applies to non-reductive representations. They have been largely excluded from our discussion because harmonic maps and maximal surfaces do not exist for these representations.
\end{remark}

\end{subsection}

\begin{subsection}{Theorem C}
Here we give the proof of Theorem C. We make use of results from the paper \cite{GK}. Fix reductive representations $\rho_1,\rho_2:\Gamma\to\PSL(2,\mathbb{R})$ with $\rho_1$ Fuchsian. 
\begin{defn}
Let $V\subset \mathbb{H}$ be a $\rho_1$-invariant domain, and $f:V\to \mathbb{H}$ a $(\rho_1,\rho_2)$-equivariant map realizing the minimal Lipschitz constant $L$ among equivariant maps. The stretch locus is the set of points $x\in \mathbb{H}$ such that the restriction of $f$ to any neighbourhood of $x$ has Lipschitz constant exactly $L$ and no smaller.
\end{defn}
The result below is culled from \cite[Theorem 1.3 and 5.1]{GK}. See the reference for more general statements and details.
\begin{thm}[Gu{\'e}ritaud-Kassel]\label{maxstretch}
Assume there exists a $(\rho_1,\rho_2)$-equivariant map with minimal Lipschitz constant $L=1$, and let $E$ be the intersection of all the stretch loci among such maps. Then there exists an ``optimal" $(\rho_1,\rho_2)$-equivariant $1$-Lipschitz map whose stretch locus is exactly $E$. $E$ projects under the action of $\rho_1(\Gamma)$ to the convex core for $\rho_1$, and is either empty or the union of a lamination and $2$-dimensional convex sets with extremal points only in the limit set $\Lambda_{\rho_1(\Gamma)}\subset \partial_\infty\mathbb{H}$.
\end{thm}

\begin{proof}[Proof of Theorem C]
The equivalence between (1) and (3) is contained in Theorem A. Assuming (1) we prove (2). Take any optimal map $g$, and the map $\tilde{C}(\mathbb{H}/\rho_1(\Gamma))\to\mathbb{H}\times \mathbb{H}$ given by $p\mapsto (p,g(p))$. In the case that there exists a peripheral on which $\rho_1$ is hyperbolic, suppose for the sake of contradiction that there is a choice of $g$ so that the domain extends to give a fibration over a larger subsurface. From the other direction of Proposition \ref{equiv}, we obtain a $(\rho_1,\rho_2)$-equivariant and a locally contracting map defined on the preimage of this subsurface in the universal cover. This implies there is a peripheral $\gamma$ with $\ell(\rho_1(\gamma))>\ell(\rho_2(\gamma))$, which contradicts our original Definition \ref{asd}.

Now we prove that (2) implies (1).  Given such a domain and fibration, from Proposition \ref{equiv} we obtain a strictly $1$-Lipschitz $(\rho_1,\rho_2)$-equivariant map defined on $\tilde{C}(\mathbb{H}/\rho_1(\Gamma))$. If $\rho_1$ has no hyperbolic peripherals, then we get (1) for free and we're done. So assume there is a peripheral $\zeta$ with $\rho_1(\zeta)$ hyperbolic. Any $1$-Lipschitz map $g$ defined inside $\tilde{C}(\mathbb{H}/\rho_1(\Gamma))$ extends to a $1$-Lipschitz map of the frontier inside $\mathbb{H}$, and hence $$\ell(\rho_2(\zeta))\leq \ell(\rho_1(\zeta)).$$ We extend $g$ to all of $\mathbb{H}$ by precomposing with the $1$-Lipschitz $(\rho_1,\rho_1)$-equivariant nearest point projection onto $\tilde{C}(\mathbb{H}/\rho_1(\Gamma))$, so we know that the set of globally defined Lipschitz maps is non-empty. From Lemma 4.10 in \cite{GK} (an application of Arz{\`e}la-Ascoli), there exists an optimal $(\rho_1,\rho_2)$-equivariant Lipschitz map $g'$. As for the optimal Lipschitz constant, $g$ shows $L\leq 1$, and if $L<1$ then $\rho_1\times\rho_2$ acts properly discontinuously on the whole $\AdS$, and hence $L=1$. Applying Theorem \ref{maxstretch}, we have a stretch locus $E$. 

$E$ is contained in the intersection of the stretch loci of $g$ and $g'$. Since $g$ does not maximally stretch in the interior of $\tilde{C}(\mathbb{H}/\rho_1(\Gamma))$, $E$ is contained in the boundary of $\tilde{C}(\mathbb{H}/\rho_1(\Gamma))$. If $E$ is missing the lifts of one boundary component of $\mathbb{C}(\mathbb{H}/\rho_1(\Gamma))$, then $g'$ is strictly contracting inside the half-spaces in $\mathbb{H}$ that project to the infinite funnel bounding this component in $\mathbb{H}/\rho_1(\Gamma)$. From Proposition \ref{equiv}, we can thus find a $\rho_1\times\rho_2$-equivariant domain that yields a fibration onto the union of the convex core with this funnel, which contradicts our standing assumption. We conclude that the stretch locus is exactly these components, and hence $g'$ is an almost strictly dominating map.

For the final statement, we use the homeomorphism $\Psi=\Psi^0$ from Theorem B to parametrize the space of representations. We take the domains in $\AdS$ associated to the spacelike maximal immersions with $0$ twist parameter (any one will do). The energy domination implies that they yield proper quotients by Proposition \ref{equiv}. Since the harmonic maps for irreducible representations vary continuously with the representation (Section 5.2), so do the domains in $\AdS$. Hence, when we restrict to these classes, $\Psi$ parametrizes a deformation space of AdS $3$-manifolds.
\end{proof}

To produce more representations that give such incomplete quotients, take an almost strictly dominating pair $(\rho_1,\rho_2)$ (Theorem B shows there are many) and an optimal map $g$. To relax the condition that all boundary lengths agree, first choose a collection of peripherals, but not all of them. For each of the selected peripherals, there is a geodesic or a horocycle in $\mathbb{H}/\rho_1(\Gamma)$. We then specify a transversely intersecting geodesic arc that does not intersect any other peripheral geodesic or horocycle, and apply strip deformations to $\mathbb{H}/\rho_1(\Gamma)$ along these arcs (see \cite{DGK1}, \cite[Section 6]{S}). This gives a new hyperbolic surface whose holonomy is a Fuchsian representation $j$, and for some $\lambda<1$, a strictly $\lambda$-Lipschitz $(j,\rho_1)$-equivariant map $g'$. We can extend $g$ outside of the convex hull of the limit set by using the $1$-Lipschitz $(\rho_1,\rho_1)$-equivariant closest-point projection. We then take the composition $g\circ g'$ and the corresponding circle bundle.

With the main theorems complete, we briefly digress to discuss the topology of the quotients. The quotients naturally acquire an orientation. Since the surface is not compact, the bundle is topologically trivial: $\textrm{BU}(1)=\mathbb{CP}^\infty$ and $[\Sigma,\mathbb{CP}^\infty]=H^2(\Sigma,\mathbb{Z})=0$.

However, the global trivialization is by no means compatible with the AdS structure. To be precise, the $3$-manifold is not ``standard" in the sense of \cite{KR}: its casual double cover does not possess a timelike Killing field. If it did, the holonomy would normalize the isometric flow generated by the Killing field, and it follows from \cite[pages 237-238]{KR} that this is impossible for reductive $\rho_2$.
\end{subsection}

\begin{subsection}{Parabolic Higgs bundles}
In \cite{AL}, Alessandrini and Li use Higgs bundles to construct AdS structures on closed $3$-manifolds. They build circle bundles explicitly, rather than first passing through Theorem 1.3. In \cite[Section 7]{S}, we followed their work closely to construct AdS structures on $3$-manifolds, which we now recognize as the same structures from Theorem C.

Since it would take us really far afield, we do not repeat the construction here. We'll just make some comments. More details and justifications for the discussion are contained in the author's PhD thesis.

As we are working over surfaces with punctures, one should use parabolic Higgs bundles (see \cite{Si}). An equivariant harmonic map gives a Higgs bundle, and the residue of the Hopf differential is encoded in the residue of the Higgs field (see \cite[page 117]{Si}). 
\begin{remark}
The Higgs bundle construction is actually more convenient if one uses the double cover of $\AdS$ as the model space. This introduces a few subtleties that we ignore below.
\end{remark}
In \cite[Section 3.5]{B}, Baraglia finds a correspondence between real projective structures on circle bundles over surfaces and equivariant maps into Grassmanians. In the case of a circle bundle with an $\textrm{AdS}^3$ structure--equivalently a real projective structure whose developing map image lies in the Klein model of $\AdS$--the equivariant map takes its image in the space of timelike $2$-planes of $\mathbb{R}^4$, otherwise known as the timelike Grassmanian $\textrm{Gr}^t(2,4).$ Equipped with its pseudo-Riemannian metric obtained through the Pl{\"u}cker embedding (see \cite[Section 7]{AL}, for instance), $\textrm{Gr}^t(2,4)$ identifies isometrically with $(\mathbb{H}\times \mathbb{H},\sigma\oplus(-\sigma)).$ Under this identification, Baraglia's construction can actually be translated to a different proof of Proposition 6.1. 

For the parabolic Higgs bundles used in the construction from \cite{S}, the equivariant mapping to $\textrm{Gr}^t(2,4)$ identifies with none other than the maximal surface from Theorem A. This does require a proof, and does not follow immediately from work contained in \cite{AL}, but it is a basic Higgs bundle computation that we omit here. When we build the circle bundles with AdS structures using Higgs bundles, the fibration in \cite{AL} is seen (via the map to the Grassmanian) to be the same as that of \cite{DT}, and the one in our previous work \cite{S} is the fibration from Theorem C.

Following the local computations from \cite{AL}, the Higgs bundles can be used to write down the AdS metric $g=(g_{ij})$ in terms of the harmonic maps data, and also to compute the volume. Explicitly, one arrives at the formula
\begin{equation}\label{volume}
    \textrm{Vol}(g_{ij})=  \pi\int_\Sigma (J(h) + J(f))dx\wedge dy,
\end{equation}
 where $h$ and $f$ are the associated harmonic maps. Since infinite energy harmonic maps converge exponentially to projections onto a geodesic, both of these Jacobians can be integrated over the surface. They also satisfy the requirements from \cite{KM}, so that the integral is the volume of the representation, depending only on the representation and computed via the relative Euler number and the total rotation (see \cite{BIW} for details).
When the $\AdS$ geometric structure is complete, this recovers Tholozan's formula for the volume of $\AdS$ quotients (see \cite[Theorem 1]{ThV}).  We leave it to the interested reader to check, but it turns out that we can also arrive at the formula (\ref{volume}) above by imitating Tholozan's integration over the fibers in \cite{ThV}.

\end{subsection}

\end{section}

\bibliographystyle{plain}
\bibliography{bibliography}

\end{document}